\documentclass[11]{article}
\usepackage[T1]{fontenc}
\usepackage{amsmath}
\usepackage{amsthm}
\usepackage{thmtools}
\usepackage{thm-restate}
\usepackage{amssymb}
\usepackage{bbm}
\usepackage{mathtools}
\usepackage{hyperref}
\usepackage{cleveref}
\usepackage{color}
\usepackage{framed}
\usepackage{enumitem}
\usepackage{xfrac}
\usepackage[tracking=true]{microtype}
\usepackage{color,soul}
\usepackage[framemethod=TikZ]{mdframed}
\usepackage[a4paper, margin=1in]{geometry}
\usepackage{xargs}                      
\usepackage[colorinlistoftodos,prependcaption,textsize=tiny]{todonotes}
\usepackage{mathnotation}

\usepackage{subcaption} 
\usepackage{float} 

\usetikzlibrary{decorations.pathreplacing,calligraphy}

\newtheorem{theorem}{Theorem}[section]
\newtheorem*{theorem*}{Theorem}
\newtheorem{corollary}{Corollary}[theorem]
\newtheorem{lemma}[theorem]{Lemma}
\newtheorem{proposition}[theorem]{Proposition}
\newtheorem{claim}[theorem]{Claim}
\newtheorem{definition}[theorem]{Definition}
\newtheorem{notation}[theorem]{Notation}

\graphicspath{{./drawings/}}

\hypersetup{colorlinks=true}	

\newcommand{\eqdef}{\vcentcolon=}

\DeclareMathAlphabet{\mbb}{U}{bbold}{m}{n}

\DeclareMathOperator{\rows}{rows}

\DeclareMathOperator{\spec}{Spec}

\DeclareMathOperator{\Sp}{sp}

\DeclareMathOperator{\colsp}{colsp}
\DeclareMathOperator{\rowsp}{rowsp}

\newcommand{\RNum}[1]{\uppercase\expandafter{\romannumeral #1\relax}}

\DeclareMathOperator{\rk}{rk}
\DeclareMathOperator{\rktwo}{rk_{2}}

\newcommand{\longdash}{\text{---}}

\makeatletter
\newcommand\incircbin
{%
  \mathpalette\@incircbin
}
\newcommand\@incircbin[2]
{%
  \mathbin%
  {%
    \ooalign{\hidewidth$#1#2$\hidewidth\crcr$#1\bigcirc$}%
  }%
}
\makeatother

\def\and{%
  \end{tabular}%
  \hskip0em \relax
  \begin{tabular}[t]{c}}

\author{Gal Beniamini \\ \footnotesize The Hebrew University of Jerusalem
        \and 
        Asaf Etgar \\ \footnotesize The Hebrew University of Jerusalem
        \and
        Yael Kirkpatrick \\ \footnotesize Massachusetts Institute of Technology}

\title{The $\FF_2$-Rank and Size of Graphs}
\date{\vspace{-3ex}}

\begin{document}

\maketitle

\begin{abstract}
We consider the extremal family of graphs of order $2^n$ in which no two vertices have identical neighbourhoods, yet the adjacency matrix has rank \textit{only} $n$ over the field of two elements. A previous result from algebraic geometry shows that such graphs exist for all even $n$ and do not exist for odd $n$. In this paper we provide a new combinatorial proof for this result, offering greater insight to the structure of graphs with these properties. We introduce a new graph product closely related to the Kronecker product, followed by a construction for such graphs for any even $n$. Moreover, we show that this is an infinite family of strongly-regular quasi-random graphs whose signed adjacency matrices are symmetric Hadamard matrices. Conversely, we provide a combinatorial proof that for all \textit{odd} $n$, no twin-free graphs of minimal $\mathbb{F}_2$-rank exist, and that the next best-possible rank $(n+1)$ \textit{is} attainable, which is tight. 
\end{abstract}

\section{Introduction}

Every graph is naturally associated with a symmetric Boolean matrix encoding its adjacency operator. The rank of this adjacency matrix, and its relation to other measures, are well-studied. For instance, determining the largest possible gap between the \textit{chromatic number} and \textit{rank} of a graph is known \cite{lovasz1988lattices} to be equivalent to the famous log-rank conjecture of communication complexity (see \cite{comm_complexity_book_nk, lovett2014recent}). This question remains wide open, despite its centrality, with an exponential gap separating the best-known upper \cite{lovett2016communication, kotlov1997rank} and lower \cite{goos2018deterministic, kushilevitz1994} bounds. A related question is comparing the \textit{rank} and \textit{order} of a graph. In this context, we restrict the discussion to \textit{twin-free} graphs (i.e., where no two vertices have the same neighbourhoods), since twins contribute to the order without affecting the rank. Kotlov and Lov{\'a}sz \cite{kotlov1996rank} showed that any twin-free graph of order $2^n$ must have \textit{real} rank at least $2n - \mathcal{O}(1)$, which is tight. We revisit this question, substituting \textit{real} rank for the rank over $\mathbb{F}_2$. Since the proof of \cite{kotlov1996rank} does not generalize to finite fields, it is natural to ask whether an equivalent result holds even over $\mathbb{F}_2$ \footnote{Achieving a \textit{linear} separation is trivial. For instance, arranging the vectors of a subspace $V \le \mathbb{F}_2^{2^n}$ of dimension $\dim(V) = n$ as the \textit{biadjacency} matrix of a bipartite graph, yields a twin-free graph of $\mathbb{F}_2$-rank $2n$ and order $2^{n+1}$.}. 

This question was answered by Godsil and Royle in \cite{godsil2001algebraic}. Using tools from symplectic flows and polarity graphs, they showed that the binary rank of the adjecncy matrix of a graph is always even (e.g. \cite[Theorem 8.10.1]{godsil2001algebraic}). Furthremore, they showed that every twin-free graph with no isolated vertices of rank $2r$ is an induced subgraph of the symplectic graph $Sp(2r)$. In particular, this proves that any twin-free graph of size $2^n$ and $\FF_2$ rank $n$ is unique.


Building upon this result, we provide an alternative combinatorial proof. Our first result is constructing a family of twin-free graphs of order $2^n$ and $\FF_2$-rank $n$ for all even $n$. We note that the uniqueness result given by the algebraic proof of Godsil and Royle can also be obtained from our construction.

\begin{restatable}{thm}{familylowrank}
    \label{thm:infinite_family_twinfree_lowf2rank}
	For every even $n$, there exists a twin-free graph $G$ such that $v(G) = 2^n$ and $\rktwo{\left(G\right)}=n$. 
\end{restatable}

A key component in our proof of \autoref{thm:infinite_family_twinfree_lowf2rank} is the ``Kronecker Parity Product'', a new graph operator. The parity product $G \boxplus H$ is a graph over the vertices $V(G) \times V(H)$, in which $(a,x) \sim (b,y)$ if and only if \textit{exactly one} of $a \sim b$ and $x \sim y$ holds. Unlike the Kronecker Product, which is well-known to be \textit{multiplicative} with respect to the rank (over any field), the parity product is only \textit{sub-additive} with respect to $\mathbb{F}_2$-rank. Moreover, under certain mild conditions, the product of two twin-free graphs remains twin-free (\textit{and} retains the necessary invariants). These properties, together with a simple base-case for $n=2$, allow us to construct an infinite family of twin-free graphs of order $2^n$ and $\mathbb{F}_2$-rank $n$, for all even $n$. Our family of graphs is extremal in several regards. For instance, for every pair of vertices\footnote{Other than the unique isolated vertex, which exists as the adjacency matrix of any twin-free graph of order $2^n$ and $\mathbb{F}_2$-rank $n$ corresponds to a subspace $V \le \mathbb{F}_2^{2^n}$ of dimension $n$, whose vectors can be arranged into a symmetric matrix with an all-zero diagonal.} in every such graph, the pairwise intersections between their neighbourhoods and their complements, partition the vertex set evenly (into quarters). Consequently, we obtain a family of strongly-regular quasi-random graphs \cite{brouwer2012strongly,hubaut1975strongly, chung1989quasi}, whose $\{\pm 1\}$-\textit{signed} adjacency operators are symmetric Hadamard matrices with a constant diagonal \cite{hedayat1978hadamard, wallis1972hadamard}. 

Our second result concerns the case of odd $n$ where we provide a combinatorial proof that \textit{no} such graphs exist.

\begin{restatable}{thm}{nooddn}
	\label{thm:no_odd_n}
	For every odd $n$, there are no twin-free graphs $G$ of order $v(G) = 2^n$ and $\rktwo{\left(G\right)}=n$.
\end{restatable}

For $n=1$ we have either two isolated vertices (which are twins), or a single edge (which is of rank 2), so clearly no such graphs exist. For $n=3$ it is similarly possible to show, through case-analysis, that no constructions are possible. Extending this claim to all odd $n$ is the main technical part of our paper, and the proof is rather involved. Our approach relies heavily on the symmetries of adjacency matrices whose rows and columns are precisely the vectors of a linear subspace of $\mathbb{F}_2^{2^n}$, of dimension $n$. In particular we show that, up to permutation, such matrices have a certain block-structure -- a fact which plays a key role in our proof. Finally, we remark that the parity-product construction of \autoref{thm:infinite_family_twinfree_lowf2rank} implies the existence of twin-free graphs of order $2^n$ and rank $(n+1)$, for all \textit{odd} $n$, which is tight.

\section{Preliminaries}
Throughout this paper, all graphs are simple and undirected. We use the following standard definitions and notation. For a graph $G$, we denote its vertex and edge sets by $V(G)$ and $E(G)$, respectively, and their cardinalities by $v(G)$ and $e(G)$. For a vertex $v$, we denote its neighbour-set by $N(v)$. If $v$ and $u$ are two adjacent vertices in $G$, we write $v\sim u$. The adjacency matrix of a graph $G$ over a field $\mathbb{F}$ is the matrix $A_G\in \mathbb{F}^{|V| \times |V|}$, where $(A_G)_{u,v} = \mathbbm{1}_{u\sim v}$. The spectrum of $G$ is defined as the spectrum of its adjacency matrix \textit{over the reals}. The rank of $G$ over $\mathbb{F}$ is defined to be the rank of its adjacency matrix over $\mathbb{F}$, namely $\rk_{\mathbb{F}}(G) \eqdef \rk_{\mathbb{F}}(A_G)$. In particular, we denote the rank over $\mathbb{F}_2$ by $\rktwo$.

When referring to a vector $u\in \FF_2^n$, we donate by $\bar{u}$ its negation. For $b\in \{0,1\}$, we write $u +b$ to mean $u + b^n$, adding $b$ at every coordinate. We denote by $u^{-1}(0)$ the coordinates at which $u$ has zeros and similarly by $u^{-1}(1)$ the coordinates at which $u$ has ones. We sometimes denote the all-zero and all-one vectors by $\mbb{1}$ and $\mbb{0}$, respectively, and specify their dimension when not clear from context. 

Finally, let us provide the following definitions used throughout this paper, relating to common families of matrices and graphs.

\begin{definition}(Hadamard Matrix)
\label{def:hadamard_matrices}
     A matrix $H\in M_{n}(\cbk{\pm 1})$ is a \textbf{Hadamard Matrix} if its rows are mutually orthogonal. Equivalently, if $HH^\top = n I_n$.
\end{definition}

\begin{definition} (Line Graph) 
     The \textbf{Line Graph} of a graph $G$ is the graph $\mathcal{L}(G)=(E(G), \hat{E})$, where: 
     \[ \forall (e_1, e_2) \in {E(G) \choose 2}: e_1 \underset{\mathcal{L}(G)}{\sim} e_2 \iff e_1 \cap e_2 \ne \emptyset .\]
\end{definition}

\begin{definition} (Strongly Regular Graph)
\label{def:strongly_regular_graph}
    A graph $G$ is said to be \textbf{Strongly Regular} with intersection array $[v,k,\lambda,\mu]$ if and only if $G$ is $k$-regular of order $v$, and $\forall x \ne y \in V(G)$ we have:
    \[ |N(x) \cap N(y)| = \begin{cases}
        \lambda, & x \sim y \\
        \mu, & x \nsim y
    \end{cases}. \]
\end{definition}

\begin{theorem} (Quasi-Random Graph) \cite{chung1989quasi}
\label{thm:quasi_random_graphs}
 Let $\mathbf{\Gamma} = \{\Gamma_n\}$ be an infinite sequence of graphs, where $v(\Gamma_n)=n$ and $e(\Gamma_n) = \left(p \pm o(1)\right){n \choose 2}$, for some $p = p(n) > 0$. The following are all equivalent:
 \begin{enumerate}
     \item $\Gamma_n$ is \textbf{Quasi-Random}.
     \item $\forall\: X,Y \subseteq V(\Gamma_n)$, $\big|e(X,Y) - p|X||Y|\big| = o(n^2)$.
     \item $\sum_{x,y \in V(\Gamma_n)} \abs{\abs{N(x)\cap N(y)} - p^2 n} = o(n^3)$.
     \item Let $\spec(\Gamma_n) = \lambda_1 \ge \dots \ge \lambda_n$. Then $\abs{\lambda_1 - pn} = o(n) \text{ and } \max\{|\lambda_2|, |\lambda_n|\} = o(n) $.
 \end{enumerate}
\end{theorem}
\section{Twin-Free Graphs of Minimal $\FF_2$-rank}

The statement of \autoref{thm:infinite_family_twinfree_lowf2rank} concerns the construction of twin-free graphs whose $\FF_2$-rank is minimal with respect to their order. The same problem can also be viewed as a question regarding particular subspaces of $\FF_2^{2^n}$, whose vectors can be arranged in a manner corresponding to the adjacency matrix of a graph (i.e., as symmetric matrices having an all-zero main diagonal).  \\

\hspace{-0.2in}\textbf{A Graph Problem} Is there an infinite family of graphs $\{ G_n \}$ such that:

\begin{itemize}[leftmargin=.5in]
	\item (Order $2^n$) \ $v(G_n) = 2^n$.
	\item (Twin-Free) $\forall u,v \in V(G_n):\ N(u) \ne N(v)$.
	\item ($\FF_2$-rank $n$) \ $\forall u \in V(G_n):\ \exists S \subseteq \left(V(G_n) \setminus \{u\}\right):\ \sum_{v \in S} \mathbbm{1}_{N(v)} = \mathbbm{1}_{N(u)}$ (summing over $\FF_2$).
\end{itemize} 
\vspace{0.1in}

\hspace{-0.2in}\textbf{A Subspace Problem} Is there an infinite family of subspaces $\big\{ V_n = \{v^1, \dots, v^{2^n}\} \le \FF_2^{2^n} \big\}$ such that:
\begin{itemize}[leftmargin=.5in]
	\item ($\FF_2$-rank $n$) \ $\dim(V_n) = n$.	
	\item (Symmetric) $
	\begin{bmatrix}
	\vert & & \vert \\
	v^{1} & \dots & v^{2^n}   \\
	\vert & & \vert
	\end{bmatrix}
	 =
	\begin{bmatrix}
	\text{---} & v^{1} & \text{---} \\
	& \vdots \phantom{A} & \\
	\text{---} & v^{2^n} & \text{---}
	\end{bmatrix}
	$.
	\item (Zero-Trace) $\forall i \in [2^n]: v^{i}_i = 0$. 
\end{itemize}

For very small values of $n$ it is not hard to construct such graphs. Consider the graph composed of the cycle $C_3$, together with an isolated vertex. This graph indeed satisfies the above properties, as the neighbour set of any vertex in the cycle is spanned by the $\FF_2$-sum of its two neighbour. The isolated vertex is spanned by the sum of neighbourhoods of \textit{all} vertices in the cycle. By enumeration, it is also evident that this construction is unique (for $n=2$). 

\begin{figure}[h]
	\centering
	\includegraphics[width=7cm]{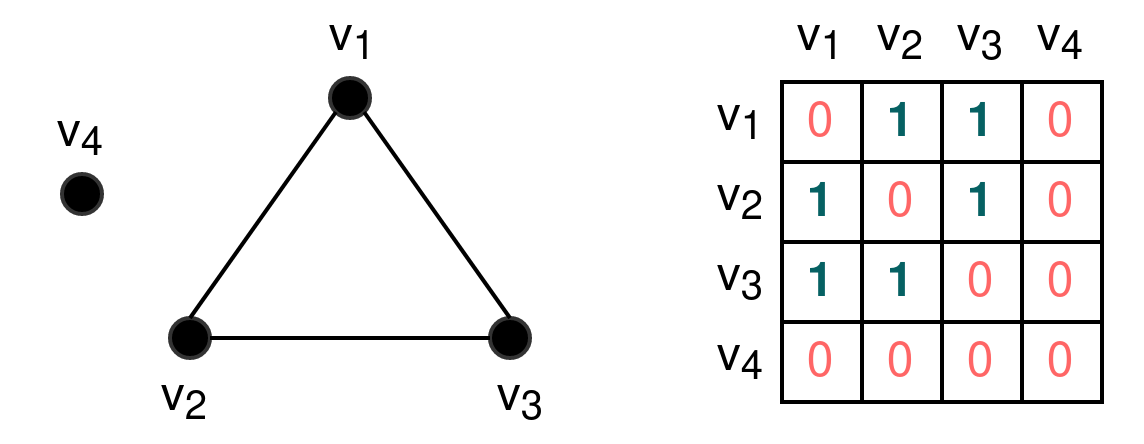}
	\caption{A twin-free graph of order $4$ and $\FF_2$-rank $2$.}
	\label{fig: twin free order four}
\end{figure}

\subsection{Warmup: Line-Graph Construction}

Constructing arbitrarily large twin-free graphs with no consideration for their $\FF_2$-rank is straightforward. For example, consider the complete graph $K_n$: it has rank roughly $n$ (its adjacency matrix is $J+I$), and is trivially twin-free. Perhaps we could transform such a graph into a new \textit{larger} twin-free graph, while also not increasing the $\FF_2$-rank by too much? To this end, let us consider its \textit{line-graph}.

%

Observe that for any graph $G$ of minimal degree $\delta(G) > 3$, the line graph $\mathcal{L}(G)$ is twin-free. Assume for sake of contradiction that this were not the case, and let $\{u,v\}, \{x,y\} \in E(G)$ be two edges such that $(u \ne x) \land (u \ne y)$. Since $\delta(G) > 3$, there exists some $w \notin \{x,y,u,v\}$ such that $u \sim w$ (in $G$). Thus, the $\{u,v\} \sim \{u,w\}$, whereas $\{u,w\} \nsim \{x,y\}$: a contradiction. Thus, the line graph $\mathcal{L}(K_n)$ for any $n > 4$ is twin-free. However, for such graphs to be good candidates, we must still compute their $\FF_2$-rank. The following claim shows that this family of graphs -- line-graphs of cliques -- exhibits a nearly quadratic separation between order and $\FF_2$-rank.

\begin{claim}($\FF_2$-rank of Line-Graphs of Cliques)
	\[\forall n > 4:\ \rktwo \left(\mathcal{L}(K_n)\right) \le \begin{cases}
		n-2, & n \equiv 0 \pmod 2,\\
		n-1, & n \equiv 1 \pmod 2.
	\end{cases} \]
\end{claim}

\begin{proof}
	Denote $V\left(\mathcal{L}(K_n)\right) = \left\{\{i,j\} \in {[n] \choose 2}\right\}$ and let $B = \left\{\{i,i+1\} : i \in [n-1]\right\} \subset V\left(\mathcal{L}(K_n)\right)$. To begin, we claim that $B$ $\FF_2$-spans all other vertices of $\mathcal{L}(K_n)$, and thus $\rktwo\left(\mathcal{L}(K_n)\right) \le |B| = n-1$. Let $\{i,j\} \in {[n] \choose 2}$ such that $j - i > 1$. Then: \[\mathbbm{1}_{N\left(\{i,j\}\right)} \equiv \sum_{k=i}^{j-1} \mathbbm{1}_{N\left(\{k,k+1\}\right)} \pmod 2.\]
	Let $\{x,y\} \in {[n] \choose 2}$ such that $x < y$. By the definition of the line graph, we have:
	\[\left(\mathbbm{1}_{N\left(\{i,j\}\right)}\right)_{\{x,y\}} \equiv \mathbbm{1}\left\{ \{i,j\} \sim \{x,y\}\right\} \pmod2 \equiv |\{i,j\} \cap \{x,y\}| \pmod 2 .\]
	As for the sum $\sum_{k=i}^{j-1} \mathbbm{1}_{N\left(\{k,k+1\}\right)}$, observe that any vertex \textit{strictly} in the interval $\{i, i+1, \dots, j\}$ contributes twice, and any \textit{endpoint} contributes exactly once. Thus:
	\begin{equation*}
	    \begin{split}
	        \left(\sum_{k=i}^{j-1} \mathbbm{1}_{N\left(\{k,k+1\}\right)}\right)_{\{x,y\}} &\equiv 2 \cdot \mathbbm{1}\{ i < x < j \} +  2 \cdot \mathbbm{1}\{ i < y < j \} + \mathbbm{1}\{ x \in \{i,j\} \} + \mathbbm{1}\{ y \in \{i,j\} \} \pmod 2 \\
	        &\equiv \mathbbm{1}\{ x \in \{i,j\} \} + \mathbbm{1}\{ y \in \{i,j\} \} \pmod 2 \\
	        &\equiv |\{i,j\} \cap \{x,y\} | \pmod 2.
 	    \end{split}
	\end{equation*}
	
	Finally, we observe that when $n \equiv 0 \pmod 2$, the vectors of $B$ are linearly dependent and consequently $\dim(\Sp(B)) < |B|$ and $\rktwo\left(\mathcal{L}(K_n)\right) \le |B| - 1 = n-2$. In this case, we claim that:
	\[\sum_{i=1}^{\sfrac{n}{2}}  \mathbbm{1}_{N\left(\{2i,2i+1\}\right)} = \mbb{0}.\]
	Which holds since for any $\{x,y\} \in {[n] \choose 2}$ such that $y-x > 1$, $\{x,y\}$ is adjacent to \textit{exactly two} of the summands above -- the one intersecting with $x$ and the one intersecting with $y$ (these are not the same summand, since $y$ is not a successor of $x$). Otherwise, if $y-x = 1$, then no summands are adjacent to $\{x,y\}$. Therefore over $\FF_2$ the sum evaluates to zero, as required.
\end{proof}

\begin{figure}[h!]
	\centering
	\includegraphics[width=9cm]{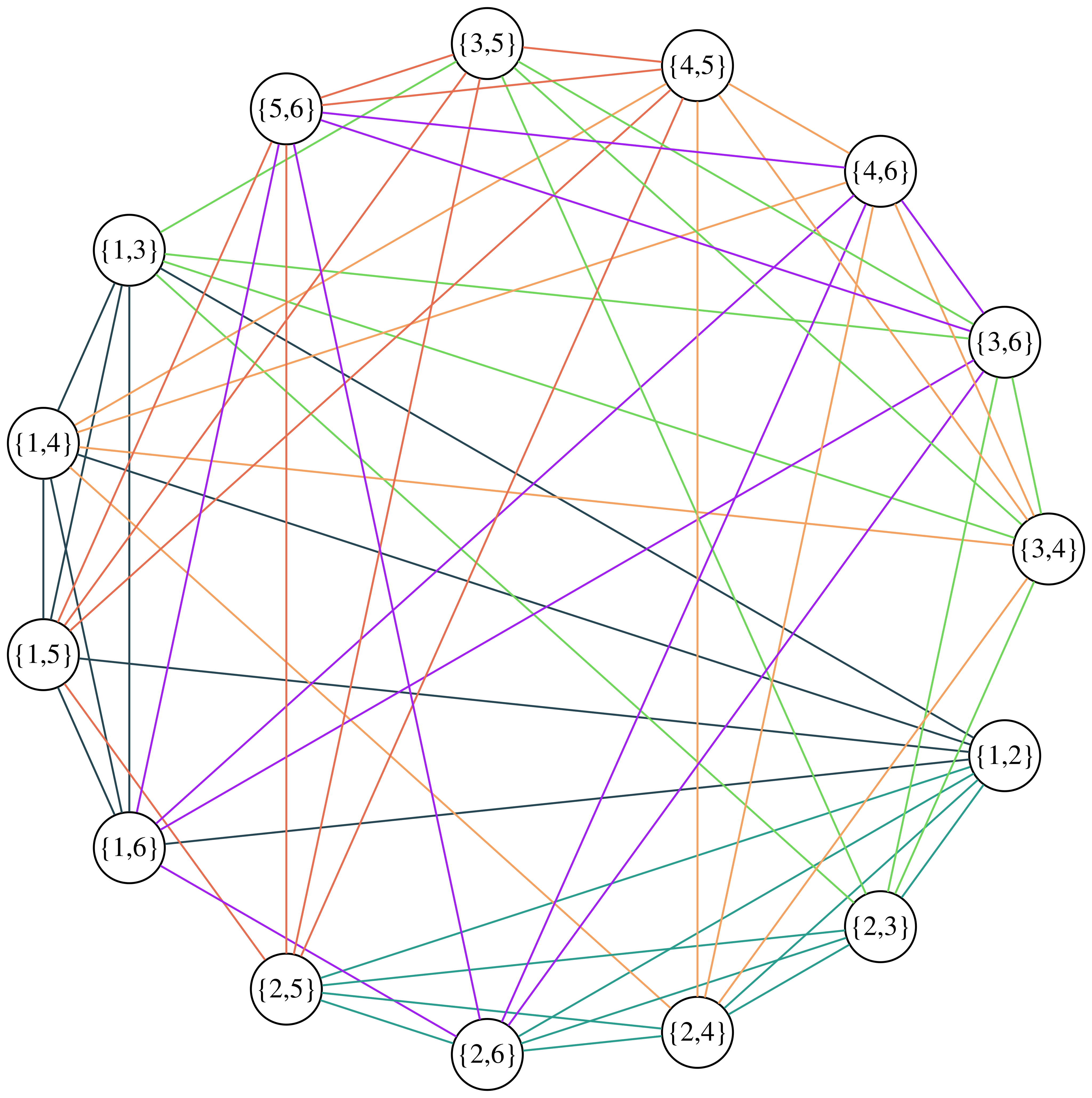}
	\caption{\centering The line-graph $\mathcal{L}\left({K_6}\right)$. Colours indicate the common vertex incident to two adjacent edges. Together with an isolated vertex, this is a twin-free graph of order $2^4$ and $\FF_2$-rank $4$, which is optimal.}
\end{figure}

Consequently, we observe that $\mathcal{L}(K_6)$ (or in other words, the Johnson graph $J(6,2)$) is a graph of order ${6 \choose 2} = 15 = 2^4 - 1$ and $\FF_2$-rank $4$. Taking this graph together with an isolated vertex does not increase the rank, and thus yields a tight construction of a twin-free graph of minimal $\FF_2$-rank, for $n=4$.

\subsection{The Construction for Even $n$}

Unlike the warm-up from the previous section, our proof of \autoref{thm:infinite_family_twinfree_lowf2rank} is based on a recursive construction. Given any two twin-free graphs $G$,$H$ of minimal $\FF_2$-rank (with respect to their order), we construct a new ``product'' graph, whose $\FF_2$ rank is \textit{sub-additive} in the ranks of $G$ and $H$, and whose order is the \textit{product} of their orders. One natural candidate for such an operator is the Kronecker product, which is defined as follows:

\begin{definition}(Kronecker Product of Graphs)
	Let $G$, $H$ be two graphs. The Kronecker Product $G \otimes H$ is defined:
	\[ V(G \otimes H) = V(G) \times V(H),\quad (u,v), (x,y) \in E(G \otimes H) \iff \left(u \underset{G}{\sim} x\right) \land \left(v \underset{H}{\sim} y\right). \]
\end{definition}

The Kronecker product can also be defined by acting on the adjacency matrices of $G$ and $H$. If $A_G \in M_n(\mathbb{R})$ and $A_H \in M_m(\mathbb{R})$ are the adjacency matrices of $G$  and $H$, respectively, where $v(G)=n$ and $v(H) = m$, then:

\[ A_{G \otimes H} = A_G \otimes A_H \eqdef \begin{bmatrix}
(a_G)_{1,1} \cdot A_H &  \dots & (a_G)_{1,n} \cdot A_H \\
\vdots & \ddots & \vdots \\
(a_G)_{n,1} \cdot A_H &  \dots & (a_G)_{n,n} \cdot A_H
\end{bmatrix} \in M_{mn}(\mathbb{R}) .\]

It is well known that for any two matrices $A$,$B$ over any field $\FF$, $\rk_{\FF}(A \otimes B) = \rk_{\FF}(A) \rk_{\FF}(B)$. However if $G$ and $H$ are two graphs of minimal $\FF_2$-rank, then to retain this property we must require that their product have rank that is only \textit{additive} in their respective ranks. Therefore, we introduce the ``Parity Product'', which is inspired by the Kronecker product, and indeed satisfies this condition (alongside other properties, allowing us to maintain the twin-free invariant).

\begin{definition}(Parity Product)
	Let $A\in M_n(\FF_2)$, $B \in M_m(\FF_2)$ be two matrices. Then:
	
	\[ A \boxplus B \eqdef \begin{bmatrix}
	(a_G)_{1,1} + A_H &  \dots & (a_G)_{1,n} + A_H \\
	\vdots & \ddots & \vdots \\
	(a_G)_{n,1} + A_H &  \dots & (a_G)_{n,n} + A_H
	\end{bmatrix} \in M_{mn}(\FF_2) .\]
	
	is the \textbf{Parity Product} of $A$ and $B$.
\end{definition}

The parity product is defined identically to the Kronecker product, if we replace the field multiplications with additions: if $\varphi$ is the group isomorphism $\left(\{0,1\}, +\right) \cong_{\varphi} \left(\{\pm 1\}, \times\right)$, then $A \boxplus B = \varphi^{-1} \left(\varphi(A) \otimes \varphi(B)\right)$, where the Kronecker product is taken over the reals. Given two graphs $G$ and $H$, we define their parity product $G \boxplus H$ to be the graph whose adjacency matrix is $A_G \boxplus A_H$. This is well defined since the product of two symmetric matrices is symmetric, and if the two matrices have zeros on their main diagonal, so does their product. The product $G \boxplus H$ can also be described as the graph operator, producing a graph over the vertices $V(G) \times V(H)$ in which $(a,x) \sim (b,y)$ if and only if \textit{exactly one} of $a \sim b$ and $x \sim y$ holds. 

To present the properties of the parity product of graphs, we require one last definition. We say that a graph $G$ is \textit{negation-free} if $\forall v \in V(G)$, there does not exist a vertex $u \in V(G)$ such that $N(v) = V(G) \setminus N(u)$. With these definitions in mind, we show the following useful lemma.

\begin{lemma}(Properties of the Parity Product)
	\label{lem:par_prod_properties}
	Let $G$ and $H$ be two graphs. Then:
	\begin{enumerate}[label=\alph*.]
		\item $\rktwo\left(G \boxplus H\right) \le \rktwo\left(G\right) + \rktwo\left(H\right)$
		\item If $G$ and $H$ are \textit{twin-free} and \textit{negation-free} then $G \boxplus H$ is also \textit{twin-free} and \textit{negation-free}.
	\end{enumerate}
\end{lemma}

\begin{proof}
	Let $A_G$ and $A_H$ be the adjacency matrices of $G$ and $H$, respectively, where:
	\begin{equation*}
		\begin{split}
			A_G = \begin{bmatrix}
				\vert & & \vert         \\
				a_{1} & \dots & a_{n}   \\
				\vert & & \vert
			\end{bmatrix} \in M_n(\FF_2)
		\end{split}
		\begin{split}
			\text{, and  }
		\end{split}
		\begin{split}
			A_H = \begin{bmatrix}
			\vert & & \vert         \\
			b_{1} & \dots & b_{m}   \\
			\vert & & \vert
			\end{bmatrix} \in M_m(\FF_2)
		\end{split}	.	
	\end{equation*}
	\underline{Proof of a.}: Denote $r_G = \rktwo(A_G)$, $r_H = \rktwo(A_H)$ and let:
	
	\[B_G = \{u_1, \dots, u_{r_G}\} \subseteq \{a_1, \dots, a_n\}, \quad B_H = \{v_1, \dots, v_{r_H}\} \subseteq \{b_1, \dots, b_m\}\]
	
	be bases for the column spaces of $A_G$ and $A_H$, respectively, and let $S \eqdef \left(B_G \boxplus 0^m\right) \cup \left(0^n \boxplus B_H\right)$. It suffices to show that $\Sp(S) \supseteq \colsp(A_G \boxplus A_H)$, since then $\rktwo(A_G \boxplus A_H) \le |S| = \rktwo(A_G) + \rktwo(A_H)$. Indeed, let $a \boxplus b$ be a column of $A_G \boxplus A_H$, where $a$,$b$ are columns of $A_G$ and $A_H$, and let us also write $a = \sum_{i \in [n]} \alpha_i u_i$ and $b = \sum_{j \in [m]} \beta_j v_j$. Then $\forall (i,j) \in [n] \times [m]$:
	
	\begin{equation*}	
	\begin{split}
		\Big((u_1 \boxplus b) + (a \boxplus v_1) + (u_1 \boxplus v_1)\Big)_{i,j} &= (u_1)_i + b_j + a_i + (v_1)_j + (u_1)_i + (v_1)_j \\
		&= a_i + b_j = (a \boxplus b)_{i,j}.
	\end{split}
	\end{equation*}
	Furthermore:
	
	\begin{equation*}
	    \begin{split}
	        u_1 \boxplus b &= \sum_{i}{\beta_i (u_1 \boxplus v_i)} = \sum_{i}{\beta_i \left((u_1 \boxplus 0^m) + (0^n \boxplus v_i)\right)} \in \Sp(S), \\
	        a \boxplus v_1 &= \sum_{i}{\alpha_i (u_i \boxplus v_1)} = \sum_{i}{\alpha_i \left((u_i \boxplus 0^m) + (0^n \boxplus v_1)\right)} \in \Sp(S), \\
	        u_1 \boxplus v_1 &= \left(u_1 \boxplus 0^m\right) + \left(0^n \boxplus v_1\right) \in \Sp(S). 
	    \end{split}
	\end{equation*}
	
	Thus $\Sp(S) \supseteq \colsp(A_G \boxplus A_H)$, as required. \\
	
	\noindent \underline{Proof of b.}: Let $a \boxplus b$, $u \boxplus v$ be two columns of $A_G \boxplus A_H$, where $a,u$ are columns of $A_G$ and $b,v$ are columns of $A_H$. Assume without loss of generality that $a \ne u$. Since $G$ is negation-free, there exist $i_1,i_2 \in [n]$ such that $a_{i_1} = u_{i_1}$ and $a_{i_2} \ne u_{i_2}$. Observe that there must exist some index $j \in [m]$ such that $b_j = v_j$. This is since either $b=v$, or if $b \ne v$, then by the negation-free property of $H$, the two vectors must agree on some coordinate. For such an index $j$, we therefore obtain:
	\begin{equation*}
		\begin{split}
		(a \boxplus b)_{i_1,j} &= a_{i_1} + b_j = u_{i_1} + v_j = (u \boxplus v)_{i_1, j}, \\
		(a \boxplus b)_{i_2,j} &= a_{i_2} + b_j \ne u_{i_2} + v_j = (u \boxplus v)_{i_2, j},
		\end{split}
	\end{equation*}
	and so $\left(a \boxplus b\right)$ and $\left(u \boxplus v\right)$ are neither negations of one another, nor are they twins, as they differ in at-least one coordinate, and agree on at-least one coordinate.
\end{proof}

The proof of \autoref{thm:infinite_family_twinfree_lowf2rank} now follows.

\familylowrank* 
\begin{proof}
	Let $G_2$ be the union of the cycle $C_3$ and an isolated vertex. $G_2$ is a \textit{twin-free} and \textit{negation-free} graph of order $2^n$ and $\FF_2$-rank $n$, for $n=2$. Consider the following family of graphs:
	
	\[ \mathcal{G} \eqdef \{G_2^{\boxplus m}\}_{m \in \mathbb{N}}, \quad \text{where } G_2^{\boxplus m} \eqdef \underbrace{G_2 \boxplus \dots \boxplus G_2}_{m \text{ times}} .\]
	
	Since $G_2$ is both negation-free \textit{and} twin-free, then from \autoref{lem:par_prod_properties} it follows that any $G_2^{\boxplus m}$ is similarly \textit{twin-free} (and negation-free). From the sub-additivity of the $\FF_2$-rank, we immediately obtain that $\rktwo(G_2^{\boxplus m}) \le \rktwo(G_2) + \dots \rktwo(G_2) = 2m$. Finally, since $v\left(G_2^{\boxplus m}\right) = 2^{2m}$, we can in fact deduce that equality holds (i.e., $\rktwo(G_2^{\boxplus m}) = 2m$), as otherwise the graph would not be twin-free.
\end{proof}


\subsection{Properties of The Family $\big\{G_2^{\boxplus n}\big\}_n$}
\label{subsection:G_n_family_properties}

Let us briefly explore some interesting properties exhibited by our construction. In what follows, let $\mathcal{G} = \big\{G^n = G_2^{\boxplus n}\big\}_n$ and let $A_{n} \in M_{N}(\{0,1\})$ be the adjacency matrix of $G^n$, where $N=2^{2n}$. Moreover, for any $\{0,1\}$-matrix $A$ we denote its $\{ \pm 1 \}$-\textit{signed} representation by $\varphi(A)$, where $\varphi$ is the isomorphism $\left(\{0,1\}, +\right) \cong_{\varphi} \left(\{\pm 1\}, \times\right)$. 


\begin{claim}
The matrices $\cbk{\varphi(A_n)}_{n\in\NN}$ are symmetric Hadamard matrices (see \autoref{def:hadamard_matrices}).
\end{claim}
\begin{proof}
    Note that 
    \[
    \varphi(A_1) = \begin{bmatrix}
    1 & -1 & -1 & 1\\
    -1 & 1 & -1 & 1\\
    -1 & -1 & 1 & 1\\
    1 & 1 & 1 & 1
    \end{bmatrix}
    \]
    is clearly a symmetric Hadamard matrix, and since $\varphi(A_n) = \varphi(A_1) \otimes \ldots\otimes \varphi(A_1)$ (and both symmetry and Hadamard-ness are closed under the Kronecker product), the claim follows.
\end{proof}

In addition to the Hadamard property, the matrices $A_n$ also satisfy the following \textit{regularity} conditions.

\begin{claim}
\label{claim:a_n_balanced_rows}
    Any non-zero row $x$ of $A_n$ is \textit{balanced}. That is, $|x^{-1}(0)| = |x^{-1}(1)| = \frac{N}{2}$.
\end{claim}
\begin{proof}
    We proceed by induction on $n$. For $n=1$ this can be verified by hand. For $n>1$, let $A_n = A_1 \boxplus A_{n-1}$ and let $\mbb{0} \ne x \in \rows(A_n)$. By construction, $x = y \boxplus z$ where $y \in \rows(A_1)$ and $z \in \rows(A_{n-1})$. Since $x \ne \mbb{0}$, one of $y,z$ is not $\mbb{0}$, and by the induction hypothesis it is a balanced vector. The claim now follows, since the parity product of a balanced vector with any another vector is balanced.
\end{proof}
\begin{claim}
\label{claim:balanced_pairwise_intersections}
    Let $x\neq y$ be two non-zero rows of $A_n$. Then, $\forall\: b,b'\in \{0,1\}$ we have $\abs{x^{-1}(b)\cap y^{-1}(b')} = \frac{N}{4}$.
\end{claim}
\begin{proof}
    We proceed once again by induction on $n$. For $n=1$ this can be verified. Let $n > 1$ and let $x,y \in \rows(A_n) \setminus \{ \mbb{0} \}$. By construction, $x = u \boxplus \alpha$ and $y = v \boxplus \beta$, where $u,v \in \rows(A_{n-1})$ and $\alpha,\beta \in \rows(A_1)$. In what follows, let us denote $\mbb{1} = 1^{\sfrac{N}{4}}$ and $\mbb{0} = 0^{\sfrac{N}{4}}$.
    \begin{enumerate}
        \item if $u = v = \mbb{0}$, then $x,y\in \cbk{(\mbb{0},\mbb{1},\mbb{1},\mbb{0}),(\mbb{0},\mbb{1},\mbb{0},\mbb{1}),(\mbb{0},\mbb{0},\mbb{1},\mbb{1})}$,
        which satisfy the claim.
        \item If $u=\mbb{0}, v \ne \mbb{0}$, then \begin{align*}
             & x\in \cbk{(\mbb{0},\mbb{1},\mbb{1},\mbb{0}),(\mbb{0},\mbb{1},\mbb{0},\mbb{1}),(\mbb{0},\mbb{0},\mbb{1},\mbb{1})}, \\
             & y\in \cbk{(v,v,v,v),(v,v+\mbb{1},v+\mbb{1},v),(v,v+\mbb{1},v,v+\mbb{1}),(v,v,v+\mbb{1},v+\mbb{1})}.
        \end{align*}
        By \autoref{claim:a_n_balanced_rows}, $v$ is balanced (and so is $v+\mbb{1}$). Thus, $y$ restricted to the indices of either $x^{-1}(0)$ or of $x^{-1}(1)$ contains exactly two copies of balanced vectors and the claim holds. 
        \item If $u,v\ne \mbb{0}$, then
        \begin{align*}
             & x\in \cbk{(u,u,u,u),(u,u+\mbb{1},u+\mbb{1},u),(u,u+\mbb{1},u,u+\mbb{1}),(u,u,u+\mbb{1},u+\mbb{1})},\\
             & y\in \cbk{(v,v,v,v),(v,v+\mbb{1},v+\mbb{1},v),(v,v+\mbb{1},v,v+\mbb{1}),(v,v,v+\mbb{1},v+\mbb{1})}.
        \end{align*}
        By the induction hypothesis, the claim holds in every quarter of $x,y$, and so it holds for the entire vectors as well. $\qedhere$
    \end{enumerate}
\end{proof}

The aforementioned properties imply several corollaries, applying both directly to $\mathcal{G}$, and to the closely related family of graphs $\mathcal{H} = \{ H^n = G^n \setminus \{\mbb{0}\} : G^n \in \mathcal{G} \}$ formed by removing the unique isolated vertex $\mbb{0}$ from every $G^n$. For instance, they allow us to precisely compute the spectrum of $G^n$, for any $n$. Moreover, they directly imply that $\mathcal{H}$ is a family of strongly regular quasi-random graphs (see \autoref{def:strongly_regular_graph} and \autoref{thm:quasi_random_graphs}). We provide short proofs of these corollaries below.

\begin{corollary}
    For any $G^n \in \mathcal{G}$ we have:
    \[ \spec(G^n) = \left\{ 0^{(1)},{\tfrac{N}{2}}^{(1)}, {\tfrac{\sqrt{N}}{2}}^{(a)}, -{\tfrac{\sqrt{N}}{2}}^{(b)} \right\} \text{, where } a = \tfrac{N + \sqrt{N}}{2} - 1 , b = \tfrac{N - \sqrt{N}}{2} - 1 \text{.} \]
\end{corollary}
\begin{proof}
Let us begin by analysing the spectrum of $\varphi(A_n)$. Since $\varphi(A_n)$ is a real symmetric matrix with $\varphi(A_n)\varphi(A_n)^\top = NI_{N}$, its eigenvalues are all precisely $\pm\sqrt{N}$. Furthermore, since $A_n$ is an adjacency matrix, the trace of $\varphi(A_n)$ (which is also the sum of its eigenvalues) is $N$. Thus:
\[ \spec(\varphi(A_n)) = \left\{ \sqrt{N} \text{ with multiplicity } \frac{(N+\sqrt{N})}{2},\ -\sqrt{N} \text{ with multiplicity } \frac{(N-\sqrt{N})}{2}\right\} .\]
Now, observe that $2A_n = J -\varphi(A_n)$, where $J$ is the all-ones matrix. Thus, up to scaling, $A_n$ differs additively from $\varphi(A_n)$ by the rank-1 matrix $J$. The following variant of Cauchy's Interlacing Theorem allows us to relate the spectra of two such matrices.
\begin{theorem*} (Cauchy Interlacing for Rank-$1$ Updates, see \cite{marcus2014ramanujan})
\label{thm:cauchy_rank_1} Let $A \in M_n(\mathbb{R})$ be a symmetric matrix, and let $v \in \mathbb{R}^n$ be a vector. Then the eigenvalues of $A$ \textit{interlace} the eigenvalues of $A + vv^\top$.
\end{theorem*}
Consequently, $\spec(A_n)$ consists of the original eigenvalues $\pm\tfrac{1}{2}\sqrt{N}$ (with one fewer multiplicity), and two additional eigenvalues $\lambda_1, \lambda_2$, where $\lambda_1 \in [-\frac{1}{2}\sqrt{N},  \tfrac{1}{2}\sqrt{N}]$ and $\lambda_2 \ge \tfrac{1}{2}\sqrt{N}$. Finally, we recall that $A_n$ is the adjacency matrix of a graph with two regular connected components: an isolated vertex, and an $\tfrac{N}{2}$-regular component. Thus, $\lambda_1 = 0$ and $\lambda_2 = \tfrac{N}{2}$, thereby concluding the proof.
\end{proof}

\begin{corollary}
\label{cor:H_n_SRG}
$H^n = G^n \setminus \{0\}$ is strongly regular with intersection array $\sbk{N-1, \frac{N}{2},\frac{N}{4},\frac{N}{4}}$.
\end{corollary}
\begin{proof}
    Removing the isolated vertex of $G^n$ does not modify neither the degrees nor co-degrees of any other vertex. Thus, from \autoref{claim:a_n_balanced_rows} we obtain the regularity condition, and from \autoref{claim:balanced_pairwise_intersections} we also obtain the condition on the co-degrees. 
\end{proof}
\begin{corollary}
\label{cor:H_n_quasi_random}
$\mathcal{H}$ is a quasi-random graph family.
\end{corollary}
\begin{proof}
    By \autoref{cor:H_n_SRG}, we have that $\forall\: x,y \in V(H^n): \abs{N(x)\cap N(y)} = \frac{N}{4}$. Thus: 
    \[
    \sum_{x,y} \abs{\abs{N(x)\cap N(y)} - \frac{N-1}{4}} = \frac{1}{4} {N \choose 2} = o(N^2).
    \]
    where $p = \tfrac{1}{4}$ is the edge density of $H^n$, and the claim now follows from property (3.) of \autoref{thm:quasi_random_graphs}.
\end{proof}


\section{Nonexistence of Extremal Graphs For Odd $n$}

\autoref{thm:infinite_family_twinfree_lowf2rank} provides a construction of $\FF_2$-rank minimal twin-free graphs, for all even $n$. It is therefore natural to ask: what about odd $n$? For $n=1$, clearly we have either two isolated vertices (which are twins) or a single edge (which is of rank 2), so no such graphs exist. Curiously, for $n=3$ we can directly show that, similarly, no such graph exists.

\begin{claim}
    \label{clm:no_twinfree_n_3}
	There is no twin-free graph of order $2^3=8$ and $\FF_2$-rank $3$.
\end{claim}

\begin{proof}
	Assume such a graph $G$ exists. $G$ has $8$ vertices, of which no two are twins, and  $\rktwo\left(A_G\right) = 3$. This implies that the rows of $A_G$ correspond to all linear combinations of some three linearly independent vectors over $\FF_2$. Denote the first three rows $v_1, v_2, v_3$ and assume w.l.o.g (by rearranging $A_G$'s rows and corresponding columns) the forth row is $v_1+v_2$, the fifth is $v_1+v_3$, the sixth $v_2+v_3$, the seventh $v_1+v_2+v_3$ and the last is all zeros. 
$$ A_G = \begin{bmatrix}
0 & x & y &  &  &  &  &  \\
 & 0 & z &  &  &  &  &  \\
 &  & 0 &  &  &  &  &  \\
 &  &  & 0 &  &  &  &  \\
 &  &  &  & 0 &  &  &  \\
 &  &  &  &  & 0 &  &  \\
 &  &  &  &  &  & 0 &  \\
0 & 0 & 0 & 0 & 0 & 0 & 0 & 0 
\end{bmatrix}  .$$

Denote by $x,y,z$ the values in the cells marked above, then we can fill in the matrix in a single way using the relation between the rows and the fact that $A_G$ is symmetric,

$$ A_G = \begin{bmatrix}
0 & x & y & x & y & x+y & x+y & 0 \\
x & 0 & z & x & x+z & z & x+z & 0 \\
y & z & 0 & y+z & y & z & y+z & 0 \\
x & x & y+z & 0 & x +y+z & x +y+z & y+z & 0 \\
y & x+z & y & x +y+z & 0 & x +y+z & x+z & 0 \\
x+y & z & z & x +y+z & x +y+z & 0 & x+y & 0 \\
x+y & x+z & y+z & y+z & x+z & x+y & 0 & 0 \\
0 & 0 & 0 & 0 & 0 & 0 & 0 & 0 
\end{bmatrix}  .$$

A simple case analysis shows that all assignments for $x,y,z$ result in a graph that is not twin-free.
\end{proof}

In fact, no such graphs exist for any odd $n$.

\nooddn*

\begin{proof}
    Assume, for sake of contradiction, that there exists a graph twin-free graph of order $2^n$ and rank $n$ over $\FF_2$, for \textit{odd} $n > 3$. Let $A\in M_{2^n}(\FF_2)$ be the adjacency matrix of this graph. Let us begin with the following simple observation:
    
\begin{proposition}
    \label{prop:rk_vs_rows}
    Let $A \in M_{m \times n}(\FF_2)$ be a matrix without duplicate rows, and let $r = \rktwo(A)$. Then $m \le 2^r$, and moreover, if $m = 2^r$ then there exist  linearly independent vectors $\{v_1, \dots, v_r\} \subseteq \FF_2^n$ such that:
    \[
        \rows(A) = \Big\{ \sum_{i=1}^r \alpha_i v_i \ :\  \alpha_1,\dots,\alpha_r \in \{0,1\} \Big\}
    \]
\end{proposition}

    Thus, the rows of $A$ form all possible linear combinations over $\FF_2$ of $n$ linearly independent vectors. Hereafter, by conjugating $A$ using the corresponding permutation matrix, let us assume that the rows (and respectively, columns) of $A$ are ordered as follows: $(\mbb{0}, v_1, v_2, v_2+v_1, v_3, v_3+v_1, \dots)$. In other words, we start with the zero subspace, and for every block of rows that follows, we take the coset formed by adding the next basis vector $v_i$ to the rows that precede it, in order. In particular, let us denote the top half of $A$'s rows by $V \le \FF_2^{2^{n-1}}$, where $\dim(V)=n-1$. Then, the bottom half are formed by $V + (u, \hat{u})$, for some pair of vectors $u,\hat{u}\in \FF_2^{2^{n-1}}$. Using the symmetry of $A$, we obtain: 
    
\begin{figure}[H]
    \centering
    \input{drawings/no_odd_drawings/TikzSettings}

\scalebox{.5}{

    \begin{tikzpicture}[scale=1, thick]

    \draw[step=4cm,black,ultra thick, -] (0,0) grid (8,8);
    \node (1) at (2,6) {{\scalebox{1.8}{${B}$}}};
    \node (1) at (6,6) {{\scalebox{1.8}{${B + U^\top}$}}};
    \node (1) at (2,2) {{\scalebox{1.8}{${B + U}$}}};
    \node (1) at (6,2) {{\scalebox{1.8}{${B + U^\top + \hat{U}}$}}};
    
    \node (01) [Transparent Orange] at (1.8,7.7) {{\scalebox{1.2}{$0\ 0\ 0\ \dots$}}};
    \node (02) [Transparent Orange] at (.3,6.6) {{\scalebox{1.2}{$0$}}};
    \node (03) [Transparent Orange] at (.3,6.15) {{\scalebox{1.2}{$0$}}};
    \node (04) [Transparent Orange] at (.3,5.7) {{\scalebox{1.2}{$0$}}};
    \node (05) [Transparent Orange] at (.3,5.25) {{\scalebox{1.2}{$\vdots$}}};



    \draw [Transparent Orange, ->] (01) edge (.25,7.7) edge (7.75,7.7)
                     (02) edge (.3,7.5)
                     (05) edge (.3,.25);
        
    
    \node [blue, opacity = .6] (V) at (-1,6) {\scalebox{1.8}{$V$}};
    \node [blue, opacity = .6] (V) at (-2,2) {\scalebox{1.8}{$V + (u,\hat{u})$}};
    
    \node [gray, opacity = .6] (V) at (2,-1) {\scalebox{1.8}{$2^{n-1}$}};
    \node [gray, opacity = .6] (V) at (6,-1) {\scalebox{1.8}{$2^{n-1}$}};
    \node [gray, opacity = .6] (V) at (9.5,2) {\scalebox{1.8}{$2^{n-1}$}};
    \node [gray, opacity = .6] (V) at (9.5,6) {\scalebox{1.8}{$2^{n-1}$}};

    \draw [decorate, decoration = {brace}, line width = 1.5pt, blue, opacity = .6] (-.3,4.1) --  (-.3,8);
    \draw [decorate, decoration = {brace}, line width = 1.5pt, blue,  opacity = .6]  (-.3,0) -- (-.3,3.9);
    
    \draw [decorate, decoration = {brace}, line width = 1.5pt, gray, opacity = .6]  (8.3,8) -- (8.3,4.1);
    \draw [decorate, decoration = {brace}, line width = 1.5pt, gray,  opacity = .6]   (8.3,3.9)  -- (8.3,0);
    \draw [decorate, decoration = {brace}, line width = 1.5pt, gray,  opacity = .6]   (3.9,-.3) -- (0,-.3) ;
    \draw [decorate, decoration = {brace}, line width = 1.5pt, gray,  opacity = .6]  (8,-.3) -- (4.1,-.3) ;
        
    \end{tikzpicture}    
}
    \begin{center}\vspace{-0.2in}where \quad\quad\quad\quad\quad\quad\quad\quad\quad\quad\quad\quad\quad\quad\quad\quad\quad\quad\quad \end{center}
    \[
        U = \begin{bmatrix}
        \longdash & u      & \longdash       \\
                  & \vdots &                 \\
        \longdash & u      & \longdash       \\
        \end{bmatrix}, \hat{U} = \begin{bmatrix}
        \longdash & \hat{u}   & \longdash    \\
                  & \vdots    &              \\
        \longdash & \hat{u}   & \longdash    \\
        \end{bmatrix}.
    \]
%
    \caption{The matrix $A$ and its submatrices.}
    \label{fig:my_label}
\end{figure}
    \noindent \vspace{-0.35in}
    \begin{claim}
        $u = \hat{u}$.
    \end{claim}
    \begin{proof}
        Assume otherwise, and let $i$ be an index such that $u_i \neq \hat{u}_i$. Then, for the bottom-right block,
        \[
            (B^\top + U^\top + \hat{U})_{i,i} = B_{i,i} + u_i + \hat{u}_i = 1,
        \]
    which is a contradiction, since $A$ has only zeros on its main diagonal (it is an adjacency matrix). $\qedhere$
    \end{proof}
    
    Next, we claim that $u$ is not contained the space spanned by the rows of $B$.
    
    \begin{lemma} \label{lem:u_notin_rowsp_b}
     $u \notin \rowsp(B)$.
    \end{lemma}
    \begin{proof}
    Assume, for sake of contradiction, that it were. First, we address the case where $B$ is twin free.  In this case, since $B$ is a $2^{n-1}\times 2^{n-1}$ matrix of rank $\rktwo(B) \leq \rktwo(B \mid B + U^\top) = n-1$,  we deduce by \autoref{prop:rk_vs_rows} that its row space is equal to the collection of its rows and hence $u\in \rows(B)$. Denote by $i_u$ the index of the row at which $u$ appears in $B$, corresponding to the row $(u, u + u_{i_u})$ of $A$, appearing at the same index. Thus, the $(2^{n-1} + i_u)$-th row of $A$ is $(u + u, u + u + u_{i_u}) = (\mbb{0}, \mbb{1}_{u_{i_u}})$. Note that $u_{i_u} = 0$ since it lies on the diagonal of $B$, and hence $A$ contains an additional all-zeros row, in contradiction to it being twin free.
    
    We are left with the case in which $B$ is not twin free. In this scenario, $B$ contains two identical rows $w$ at some indices $a,b$. These correspond to the rows $(w, w + u_a), (w, w + u_b)$ in $A$. Since $A$ is twin free, $u_a \neq u_b$. Using the assumption $u \in \rowsp(B)$, we may write $u = \sum_{i\in S} v_i$ where $S \subseteq [2^{n-1}]$ and $\{v_i\}_{i\in S} \subseteq \rows(B)$. Subsequently, we claim that any row in the bottom half of $A$ is contained in the span of its top half, implying $\rktwo(A) \leq n-1$. Let $(x + u, x + u + u_{i_x})$ be a row in the bottom half of $A$ at index $2^{n-1} + i_x$. Then,
    \begin{align*}
        (x + u, x + u + u_{i_x}) = 
        (x, x + u_{i_x}) + \underbrace{\sum_{i\in S}(v_i, v_i + u_i)}_{=(u, u + \sum_{i\in S}u_i)} +
        \sum_{i\in S} u_i \cdot \underbrace{\left[(w, w + u_a) + (w, w + u_b) \right]}_{=(0, \ldots, 0, 1, \ldots, 1)}.
    \end{align*}
    Note that all the above vectors are rows in the top half of $A$. Thus $\rktwo(A) \leq n-1$, contradiction. $\qedhere$
    \end{proof}
    
    \begin{lemma} \label{lem:rank_b_n_minus_2}
    $\rktwo(B) = n - 2$.
    \end{lemma}
    \begin{proof} Consider the following inequalities. \\
    
    \underline{$\rktwo(B) \geq n - 2$:} Assuming the contrary; $\rktwo(B) < n-2$. Let $u^{-1}(0)$ be the coordinates at which $u$ has zeros, and $u^{-1}(1)$ be the coordinates at which $u$ has ones. Without loss of generality, $|u^{-1}(0)| \geq 2^{n-2}$ and so the matrix $B$ restricted to rows $u^{-1}(0)$ has $\geq 2 ^{n-2}$ rows and rank $< n-2$, thus by \autoref{prop:rk_vs_rows} this matrix is not twin free and has at least two identical rows. However, since at these indices $u$ has zeros, these correspond to two identical rows in $(B \mid B + U^\top)$ and thus to identical rows in $A$, contradiction. \\
    
    \underline{$\rktwo(B) \le n-2$:} Clearly $\rktwo(B) \leq \rktwo(B \mid B + U^\top) = n-1$. Assume towards a contradiction that $\rktwo(B) = n-1$. Then, from \autoref{lem:u_notin_rowsp_b} we know $u\notin \rowsp(B)$ and we will argue that this implies $\rktwo(B + U) > n-1$. Note that this, in turn, would imply a contradiction, as
    \[
        n-1 = \rktwo(B \mid  B + U^\top) \geq \rktwo(B + U^\top) = \rktwo(B + U) > n-1.
    \]
    where the first inequality follows since the restriction cannot increase the rank, and the second equality follows since $B$ is symmetric. To show $\rktwo(B + U) > n-1$, let $\left\{v_1, \ldots, v_{n-1}\right\} \subseteq \rows(B)$ be a basis for the rowspace of $B$. Then, we claim that $\left\{v_1 + u, \ldots, v_{n-1} + u, u\right\} \subseteq \rows(B + U)$ is a set of of $n$ linearly independent vectors. Indeed, for any $\emptyset \ne S \subseteq [n-1]$, we have 
    \[
        \sum_{i\in S}(v_i + u) = \sum_{i\in S}v_i + |S|\cdot u.
    \]
    Since $\sum_{i\in S}v_i$ is neither $\mbb{0}$ nor $u$, the above sum is similarly neither $\mbb{0}$ nor $u$ (since we are over $\FF_2$). Thus, adding $u$ to any of the above linear combinations similarly does not vanish, and so we have found $n$ linearly independent rows in $B + U$, which completes the proof of \autoref{lem:rank_b_n_minus_2}.$\qedhere$
    \end{proof}
    
    Since $B$ is a symmetric subspace matrix with zeros on the main diagonal, then, as we have deduced before for $A$, it can subdivided into four square blocks, in the following manner:
    \[
        B = \renewcommand{\arraystretch}{2}\left[\begin{array}{c|c}
    C  & C + W^\top \\
    \hline
    C + W & C + W + W^\top
    \end{array}\right], \text{ where } W = \renewcommand{\arraystretch}{1}\begin{bmatrix}
    \longdash & w      & \longdash \\
              & \vdots &  \\
    \longdash & w      & \longdash
    \end{bmatrix}.
    \]
    for some $w\in \FF_2^{2^{n-2}}$. Denote $u =(x,y)$ where $x,y\in \FF_2^{2^{n-2}}$ are the first and second half of the vector $u$ and  denote by $(s,t)$ the continuation of the row $w$ in the matrix $A$, as demonstrated in \autoref{fig:A_B_notation}.
    \begin{figure}[H]
        \centering
        \input{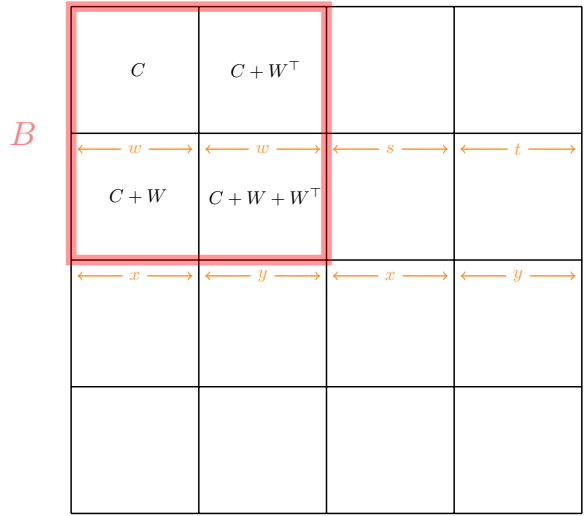}

\scalebox{.7}{

    \begin{tikzpicture}[scale=.6, thick]
\draw[step=4cm,black,thick, -] (0,0) grid (16,16);
\node (1) at (2,14) {{ {$C$}}};
\node (3) at (2,10) {{ {$C + W$}}};
\node (2) at (6,14) {{ {$C + W^\top$}}};
\node (4) at (6,10) {{ {$C + W + W^\top$}}};

\node (w1) [Transparent Orange] at (2,11.5) {$w$};
\node (w2) [Transparent Orange] at (6,11.5) {$w$};

\node (s) [Transparent Orange] at (10,11.5) {$s$};
\node (t) [Transparent Orange] at (14,11.5) {$t$};

\node (x1) [Transparent Orange] at (2,7.5) {$x$};
\node (x2) [Transparent Orange] at (10,7.5) {$x$};

\node (y1) [Transparent Orange] at (6,7.5) {$y$};
\node (y2) [Transparent Orange] at (14,7.5) {$y$};

\node (B) [red, opacity = .5] at (-1.5,12) {\scalebox{1.8}{${B}$}};

        
\draw[line width = 2mm, opacity = .4, red] (0,16) -- (8,16) -- (8,8) -- (0,8) -- cycle;
\draw[->, Transparent Orange]
    (w1) edge (.2,11.5) edge (3.8,11.5)
    (w2) edge (4.2,11.5) edge (7.8,11.5)
    (s) edge (8.2,11.5) edge (11.8,11.5)
    (t) edge (12.2,11.5) edge (15.8,11.5)
    (x1) edge (.2,7.5) edge (3.8,7.5)
    (y1) edge (4.2,7.5) edge (7.8,7.5)
    (x2) edge (8.2,7.5) edge (11.8,7.5)
    (y2) edge (12.2,7.5) edge (15.8,7.5);

    
\end{tikzpicture}    
}
        \caption{The matrix $A$ and its submatrix $B$.}
        \label{fig:A_B_notation}
    \end{figure}
    Hereafter, by convention, let us denote by $X,Y,S,T$ the matrices whose rows are $x,y,s,t$ respectively.
    \begin{claim}
    s = t.
    \end{claim}
    \begin{proof}
    Using the coset and symmetry constraints of the matrix $A$ as demonstrated in \autoref{fig:s_eq_t}, we conclude that the bottom right $2^{n-2}\times 2^{n-2}$ submatrix of $A$ is equal to $C + X + X^\top + S^\top + T$.
    \begin{figure}[H]
        \centering
        \input{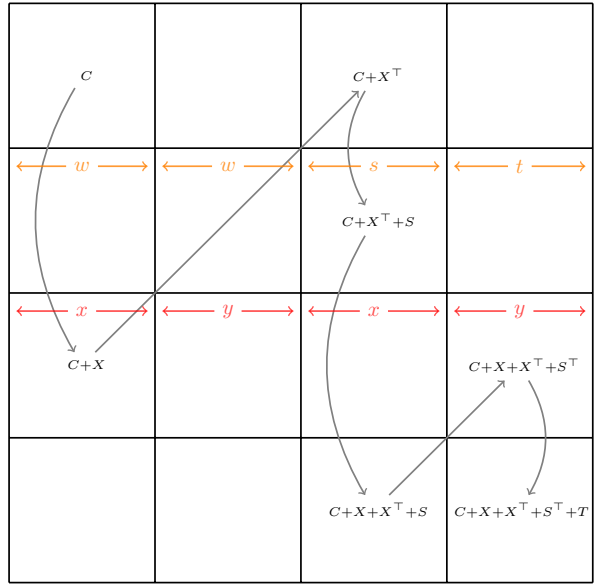}

\scalebox{.8}{

    \begin{tikzpicture}[scale=.6, ->, thick]

    \draw[step=4cm,black,thick, -] (0,0) grid (16,16);
    \node (1) at (2,14) {{ {\fontsize{6.5}{5} \selectfont $C\ $} }};
    \node (9) at (2,6) {{ {\fontsize{6.5}{5} \selectfont ${C}{+}{X}\ $} }};
    \node (5) at (10,14) {{ {\fontsize{6.5}{5} \selectfont ${C}{+}{X^\top}\ $} }};
    \node (7) at (10,10) {{ {\fontsize{6.5}{5} \selectfont ${C}{+}{X^\top}{+}{S}\ $} }};
    \node (15) at (10,2) {{ {\fontsize{6.5}{5} \selectfont ${C}{+}{X}{+}{X^\top}{+}{S}\ $} }};
    \node (14) at (14,6) {{ {\fontsize{6.5}{5} \selectfont ${C}{+}{X}{+}{X^\top}{+}{S^\top}\ $} }};
    \node (16) at (14,2) {{ {\fontsize{6.5}{5} \selectfont ${C}{+}{X}{+}{X^\top}{+}{S^\top}{+}{T}\ \ $} }};
    
    \node (w1) [Transparent Orange] at (2,11.5) {$w$};
    \node (w2) [Transparent Orange] at (6,11.5) {$w$};

    \node (s) [Transparent Orange] at (10,11.5) {$s$};
    \node (t) [Transparent Orange] at (14,11.5) {$t$};
    
    \node (x1) [Transparent Red] at (2,7.5) {$x$};
    \node (x2) [Transparent Red] at (10,7.5) {$x$};
    
    \node (y1) [Transparent Red] at (6,7.5) {$y$};
    \node (y2) [Transparent Red] at (14,7.5) {$y$};
    
    \draw [gray] (1) edge[bend right] (9)
            (9) edge  (5)
            (5) edge [bend right] (7)
            (7) edge [bend right] (15)
            (15) edge  (14)
            (14) edge [bend left] (16);
    \draw[->, Transparent Orange]
        (w1) edge (.2,11.5) edge (3.8,11.5)
        (w2) edge (4.2,11.5) edge (7.8,11.5)
        (s) edge (8.2,11.5) edge (11.8,11.5)
        (t) edge (12.2,11.5) edge (15.8,11.5)
        ;
    \draw[->, Transparent Red]
        (x1) edge (.2,7.5) edge (3.8,7.5)
        (y1) edge (4.2,7.5) edge (7.8,7.5)
        (x2) edge (8.2,7.5) edge (11.8,7.5)
        (y2) edge (12.2,7.5) edge (15.8,7.5);

    \end{tikzpicture}    
}
        \caption{Computing the bottom right submatrix of $A$.}
        \label{fig:s_eq_t}
    \end{figure}
    
    Since this submatrix lies on the main diagonal of $A$, it too has zeros on its diagonal. Therefore, 
    \begin{align*}
    0 &= (C + X + X^\top + S^\top + T)_{ii} = C_{ii} + x_i + x_i + s_i + t_i = s_i + t_i 
    \end{align*}
    and so $s_i = t_i$ for every  $i\in [2^{n-2}]$, and we conclude $s=t$. $\qedhere$
    \end{proof}
    \begin{claim}
    Either ($w=s$ and $x = y$) or ($w = \bar{s}$ and $x = \bar{y}$).
    \end{claim}
    \begin{proof}
    Using the coset and symmetry constraints of the matrix $A$ as before, we can compute the submatrix marked by $\incircbin{\star}$ in \autoref{fig:w_eq_s_x_eq_y} in two different ways.
    \begin{figure}[H]
        \centering
        \input{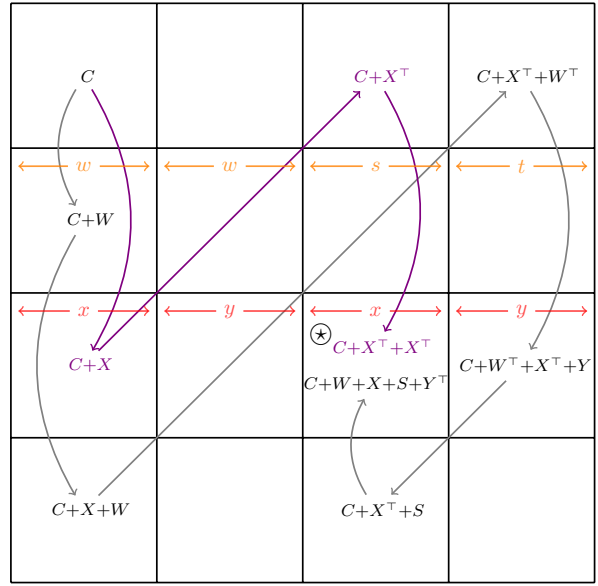}

\scalebox{.8}{

    \begin{tikzpicture}[scale=.6, ->, thick]

    \draw[step=4cm,black,thick, -] (0,0) grid (16,16);
    \node (1) at (2,14) {{ {\fontsize{7.6}{5} \selectfont $C\ $} }};
    \node (3) at (2,10) {{ {\fontsize{7.6}{5} \selectfont ${C}{+}{W}$} }};
    \node (5)[violet] at (10,14) {{ {\fontsize{7.6}{5} \selectfont ${C}{+}{X^\top}$} }};
    \node (6) at (14,14) {{ {\fontsize{7.6}{5} \selectfont ${C}{+}{X^\top}{+}{W^\top}$} }};
    \node (9)[violet] at (2,6) {{ {\fontsize{7.6}{5} \selectfont ${C}{+}{X}$} }};
    \node (11) at (2,2) {{  {\fontsize{7.6}{5} \selectfont ${C}{+}{X}{+}{W}$} }};
    \node (131) at (9.92,5.5) {{ {\fontsize{7.6}{5} \selectfont ${C}{+}{W}{+}{X}{+}{S}{+}{Y^\top}\ $} }};
    \node (132)[violet] at (10,6.5) {{ {\fontsize{7.6}{5} \selectfont ${C}{+}{X^\top}{+}{X^\top}$} }};
    \node (14) at (14,6) {{ {\fontsize{7.6}{5} \selectfont ${C}{+}{W^\top}{+}{X^\top}{+}{Y}\ $} }};
    \node (15) at (10,2) {{ {\fontsize{7.6}{5} \selectfont ${C}{+}{X^\top}{+}{S}$} }};

    \node (w1) [Transparent Orange] at (2,11.5) {$w$};
    \node (w2) [Transparent Orange] at (6,11.5) {$w$};
    
    \node (s) [Transparent Orange] at (10,11.5) {$s$};
    \node (t) [Transparent Orange] at (14,11.5) {$t$};
    
    \node (x1) [Transparent Red] at (2,7.5) {$x$};
    \node (x2) [Transparent Red] at (10,7.5) {$x$};
    
    \node (y1) [Transparent Red] at (6,7.5) {$y$};
    \node (y2) [Transparent Red] at (14,7.5) {$y$};
    
    \node (star) at (8.5,6.8) {\scalebox{1.1}{$\incircbin{\star}$}};
    
    \draw [gray] (1) edge[bend right] (3)
            (3) edge [bend right] (11)
            (11) edge (6)
            (6) edge [bend left] (14)
            (14) edge  (15)
            (15) edge [bend left] (131);
            
    \draw[violet]
        (1) edge [bend left] (9)
        (9) edge (5)
        (5) edge[bend left] (132);
    \draw[->, Transparent Orange]
        (w1) edge (.2,11.5) edge (3.8,11.5)
        (w2) edge (4.2,11.5) edge (7.8,11.5)
        (s) edge (8.2,11.5) edge (11.8,11.5)
        (t) edge (12.2,11.5) edge (15.8,11.5);
    \draw[->, Transparent Red]
        (x1) edge (.2,7.5) edge (3.8,7.5)
        (y1) edge (4.2,7.5) edge (7.8,7.5)
        (x2) edge (8.2,7.5) edge (11.8,7.5)
        (y2) edge (12.2,7.5) edge (15.8,7.5);
    
    \end{tikzpicture}    
}
        \caption{Computing the marked submatrix in two different ways.}
        \label{fig:w_eq_s_x_eq_y}
    \end{figure}
    \noindent Comparing the two results we obtained for the submatrix $\incircbin{\star}$, 
    \begin{align*}
    C + X + X^\top = C + W + X + Y^\top +S \implies X^\top + Y^\top = W + S
    \end{align*}
    Thus $x_i + y_i = w_j + s_j$, for every $i,j$.
    Note that $x_1 = w_1 = 0$ since they lie on the diagonal. Thus, plugging in $i=1, j=1$ we get $s_1 = y_1$. Furthermore, plugging in $i=1$ or $j=1$, separately, we have:
    \begin{align*}
      s_j + w_j = y_1 \ \forall\: j, \text{ and } x_i + y_i = s_1 \ \forall\: i.
    \end{align*}
    Therefore, if $s_1 = 1$ then $s = w + \one = \bar{w}$ and $x = y + \one = \bar{y}$ and otherwise $s=w, x=y$. 
    \end{proof}
    
    \begin{lemma}\label{lem:xw_rowsp_c}
    Either $x,w \in \rowsp(C)$ or $x,w\notin \rowsp(C)$.
    \end{lemma}
    \begin{proof}
    Recall we have one of the two following cases,
    \begin{figure}[H]
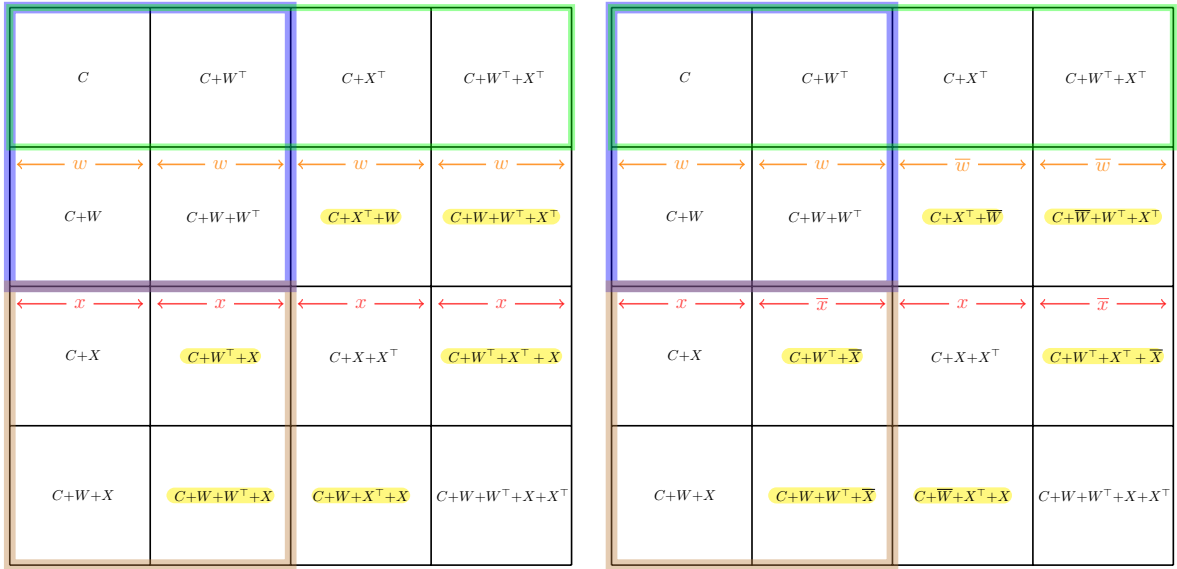

    \centering
    \begin{subfigure}{0.49\textwidth}
        \centering
        \input{drawings/no_odd_drawings/TikzSettings}

\scalebox{.77}{

    \begin{tikzpicture}[scale=.6, ->, thick]
    \draw[step=4cm,black,thick, -] (0,0) grid (16,16);
    \node (1) at (2,14) {{ \scalebox{.7}{$C$}}};
    \node (2) at (6,14) {{ \scalebox{.7}{${C}{+}{W^\top}$}}};
    \node (3) at (2,10) {{ \scalebox{.7}{${C}{+}{W}$}}};
    \node (4) at (6,10) {{ \scalebox{.7}{${C}{+}{W}{+}{W^\top}$}}};
    \node (5) at (10,14) {{ \scalebox{.7}{${C}{+}{X^\top}$}}};
    \node (6) at (14,14) {{ \scalebox{.7}{${C}{+}{W^\top}{+}{X^\top}$}}};
    \node (7) [Diff Nodes] at (10,10) {{ \scalebox{.7}{${C}{+}{X^\top}{+}{W}$}}};
    \node (8) [Diff Nodes] at (14,10) {{  \scalebox{.7}{${C}{+}{W}{+}{W^\top}{+}{X^\top}$}}};
    \node (9) at (2,6) {{ \scalebox{.7}{${C}{+}{X}$}}};
    \node (10) [Diff Nodes] at (6,6) {{ \scalebox{.7}{${C}{+}{W^\top}{+}{X}$}}};
    \node (11) at (2,2) {{ \scalebox{.7}{${C}{+}{W}{+}{X}$}}};
    \node (12)[Diff Nodes] at (6,2) {{  \scalebox{.7}{${C}{+}{W}{+}{W^\top}{+}{X}$}}};
    \node (13) at (10,6) {{ \scalebox{.7}{${C}{+}{X}{+}{X^\top}$}}};
    \node (14)[Diff Nodes] at (14,6) {{  \scalebox{.7}{${C}{+}{W^\top}{+}{X^\top + X}$}}};
    \node (15)[Diff Nodes] at (10,2) {{\scalebox{.7}{${C}{+}{W}{+}{X^\top}{+}{X}$}}};
    \node (16) at (13.95,2) {{ \scalebox{.7}{${C}{+}{W}{+}{W^\top}{+}{X}{+}{X^\top}$}}};
    
    \node (w1) [Transparent Orange] at (2,11.5) {$w$};
    \node (w2) [Transparent Orange] at (6,11.5) {$w$};
    
    \node (s) [Transparent Orange] at (10,11.5) {$w$};
    \node (t) [Transparent Orange] at (14,11.5) {$w$};
    
    \node (x1) [Transparent Red] at (2,7.5) {$x$};
    \node (x2) [Transparent Red] at (10,7.5) {$x$};
    
    \node (y1) [Transparent Red] at (6,7.5) {$x$};
    \node (y2) [Transparent Red] at (14,7.5) {$x$};
    
    \draw[->, Transparent Orange]
        (w1) edge (.2,11.5) edge (3.8,11.5)
        (w2) edge (4.2,11.5) edge (7.8,11.5)
        (s) edge (8.2,11.5) edge (11.8,11.5)
        (t) edge (12.2,11.5) edge (15.8,11.5);
    
    \draw[->, Transparent Red]    
        (x1) edge (.2,7.5) edge (3.8,7.5)
        (y1) edge (4.2,7.5) edge (7.8,7.5)
        (x2) edge (8.2,7.5) edge (11.8,7.5)
        (y2) edge (12.2,7.5) edge (15.8,7.5);
    
    \draw[line width = 2mm, opacity = .4, blue] (0,16) -- (8,16) -- (8,8) -- (0,8) -- cycle;
    \draw[line width = 1.2mm, opacity = .4, green] (0,16) -- (16,16) -- (16,12) -- (0,12) -- cycle;
    \draw[line width = 2mm, opacity = .4, brown] (0,8) -- (8,8) -- (8,0) -- (0,0) -- cycle;

    \end{tikzpicture}    
}
        \caption{Case I}
    \end{subfigure}
    \begin{subfigure}{0.49\textwidth}
        \centering
        \input{drawings/no_odd_drawings/TikzSettings}

\scalebox{.77}{

    \begin{tikzpicture}[scale=.6, ->, thick]
    \draw[step=4cm,black,thick, -] (0,0) grid (16,16);
    \node (1) at (2,14) {{ \scalebox{.7}{$C$}}};
    \node (2) at (6,14) {{ \scalebox{.7}{${C}{+}{W^\top}$}}};
    \node (3) at (2,10) {{ \scalebox{.7}{${C}{+}{W}$}}};
    \node (4) at (6,10) {{ \scalebox{.7}{${C}{+}{W}{+}{W^\top}$}}};
    \node (5) at (10,14) {{ \scalebox{.7}{${C}{+}{X^\top}$}}};
    \node (6) at (14,14) {{ \scalebox{.7}{${C}{+}{W^\top}{+}{X^\top}$}}};
    \node (7) [Diff Nodes] at (10,10) {{ \scalebox{.7}{${C}{+}{X^\top}{+}{\overline{W}}$}}};
    \node (8) [Diff Nodes] at (14,10) {{  \scalebox{.7}{${C}{+}{\overline{W}}{+}{W^\top}{+}{X^\top}$}}};
    \node (9) at (2,6) {{ \scalebox{.7}{${C}{+}{X}$}}};
    \node (10) [Diff Nodes] at (6,6) {{ \scalebox{.7}{${C}{+}{W^\top}{+}{\overline{X}}$}}};
    \node (11) at (2,2) {{ \scalebox{.7}{${C}{+}{W}{+}{X}$}}};
    \node (12)[Diff Nodes] at (6,2) {{  \scalebox{.7}{${C}{+}{W}{+}{W^\top}{+}{\overline{X}}$}}};
    \node (13) at (10,6) {{ \scalebox{.7}{${C}{+}{X}{+}{X^\top}$}}};
    \node (14)[Diff Nodes] at (14,6) {{  \scalebox{.7}{${C}{+}{W^\top}{+}{X^\top + \overline{X}}$}}};
    \node (15)[Diff Nodes] at (10,2) {{\scalebox{.7}{${C}{+}{\overline{W}}{+}{X^\top}{+}{X}$}}};
    \node (16) at (13.95,2) {{ \scalebox{.7}{${C}{+}{W}{+}{W^\top}{+}{X}{+}{X^\top}$}}};
    
    \node (w1) [Transparent Orange] at (2,11.5) {$w$};
    \node (w2) [Transparent Orange] at (6,11.5) {$w$};
    
    \node (s) [Transparent Orange] at (10,11.5) {$\overline{w}$};
    \node (t) [Transparent Orange] at (14,11.5) {$\overline{w}$};
    
    \node (x1) [Transparent Red] at (2,7.5) {$x$};
    \node (x2) [Transparent Red] at (10,7.5) {$x$};
    
    \node (y1) [Transparent Red] at (6,7.5) {$\overline{x}$};
    \node (y2) [Transparent Red] at (14,7.5) {$\overline{x}$};
    
    \draw[->, Transparent Orange]
        (w1) edge (.2,11.5) edge (3.8,11.5)
        (w2) edge (4.2,11.5) edge (7.8,11.5)
        (s) edge (8.2,11.5) edge (11.8,11.5)
        (t) edge (12.2,11.5) edge (15.8,11.5);
    
    \draw[->, Transparent Red]    
        (x1) edge (.2,7.5) edge (3.8,7.5)
        (y1) edge (4.2,7.5) edge (7.8,7.5)
        (x2) edge (8.2,7.5) edge (11.8,7.5)
        (y2) edge (12.2,7.5) edge (15.8,7.5);
    
    \draw[line width = 2mm, opacity = .4, blue] (0,16) -- (8,16) -- (8,8) -- (0,8) -- cycle;
    \draw[line width = 1.2mm, opacity = .4, green] (0,16) -- (16,16) -- (16,12) -- (0,12) -- cycle;
    \draw[line width = 2mm, opacity = .4, brown] (0,8) -- (8,8) -- (8,0) -- (0,0) -- cycle;

    \end{tikzpicture}    
}
        \caption{Case II}
    \end{subfigure}
    \caption{The two possible cases for $A$ and its submatrices, the only differences are highlighted in yellow}
    \label{fig:two_cases}
    \end{figure}
    In both cases the blue square and green rectangle in \autoref{fig:two_cases} are the same,  thus in both cases,
    $$
    n-2 = \rktwo(C \mid C + W^\top \mid C + X^\top \mid C + W^\top + X^\top ) = \rktwo(C) + \one_{x\in \Sp(C)} + \one_{w\in \Sp(C)}.
    $$
    
    We proceed by case analysis, to rule out the two undesirable cases for $x,w$. Moreover, for simplicity, let us introduce the following notation used throughout the remainder of our proof.

    \begin{notation}
        Let $A \in M_{m \times n}(\FF_2)$ be a matrix and let $S \subseteq [m]$. Then, $A_S$ is $A$ restricted to the rows $S$.
    \end{notation}

    \underline{Assume for contradiction $w\in \Sp(C), x\notin \Sp(C)$}: In what follows, let us denote $M = (C \mid C + W^\top)$, and $M_b = M_{x^{-1}(b)}$ for $b \in \{0,1\}$. By assumption, $\rktwo(C) = \rktwo(M) = n-3$. Now, notice that any twin rows in $M_b$ correspond to twin rows in $A$ (by considering their continuation within the green rectangle, as in \autoref{fig:two_cases}). However, by \autoref{prop:rk_vs_rows}, if $|x^{-1}(b)| > 2^{n-3}$, then the matrix $M_b$ necessarily has twin rows. So to avoid twin rows in $A$, we conclude that $x$ is \textit{balanced}, i.e., $|x^{-1}(0)| = |x^{-1}(1)| = 2^{n-3}$. Applying \autoref{prop:rk_vs_rows} to the matrices $M_b$ once more, we obtain that:
    \begin{align*}
        \rows(M_0) = V_0 \leq \FF_2^{2^{n-1}}, \text{ and }\rows(M_1) = V_1 \leq \FF_2^{2^{n-1}}
    \end{align*}
    where $\dim(V_0) = \dim(V_1) = n-3$. Finally, we remark that $V_0 = V_1$ since otherwise $\rktwo(M)$ would exceed $n-3$, in contradiction to our assumption. Hereafter let us denote this subspace $V$.
    
    From the assumption $w \in \Sp(C)$, it follows that there exists a set $S \subseteq \rows(C)$ whose sum is $w$. Since $C$ contains two copies of the same row set (from its two copies of $V$, restricted to the columns of $C$), it suffices to use only one such copy to span $w$. However, these sets are similarly closed under addition (since $V$ is a subspace), and hence $w \in \rows(C)$. Denote by $i_w$ the index of the row $w$, and note that $w_{i_w} = 0$ as it lies on the main diagonal. Therefore, $w + w_{i_w} = w$ and so $(w,w) \in \rows(M)$. Ergo, the rows in the top half of $B$ span its bottom half. Applying \autoref{lem:rank_b_n_minus_2} yields the following contradiction:
    \[
        n-2 = \rktwo(B) = \renewcommand{\arraystretch}{1.5}\rktwo
        \left[\begin{array}{c|c} 
        	C & C + W^\top \\ 
        	\hline 
        	C + W & C + W + W^\top  
        \end{array}\right]
        = \rktwo(C \mid C + W^\top) = n-3.
    \]
    
    \underline{Assume for contradiction $w\notin \Sp(C), x\in \Sp(C)$}: This case is very similar to the one above, and is included for completeness. Let us denote $M = (C \mid C + X^\top)$, and $M_b = M_{x^{-1}(b)}$ for $b \in \{0,1\}$. By the same reasoning as before, we have $\rktwo(C) = \rktwo(M) = n-3$, $w$ is \textit{balanced} ($|w^{-1}(0)| = |w^{-1}(1)| = 2^{n-3}$), and
    \begin{align*}
        \rows(M_0) = \rows(M_1) = V \le \FF_2^{2^{n-1}}, \text{ where } \dim(V) = n-3.
    \end{align*}

    In this case, $x\in \Sp(C) \implies x\in \rows(C)$. Using the notation of $i_x$ for the index of the row $x$, since $x_{i_x} = 0$, we have $(x,x) \in \rows(M)$. Furthermore, this row appears \textbf{twice} in $\rows(M)$, once in the rows indexed by $w^{-1}(0)$, and once in the rows indexed by $w^{-1}(1)$. Meaning, $(x,x), (x,\bar{x}) \in \rows(M)$. Therefore, in both cases I and II, the rows in the blue square span those in the brown square. As a result, the concatenation of these two squares, i.e., the left half of $A$, has the same rank as $B$, the blue square. Recall that by construction the left half of $A$ has rank exactly $n-1$, whereas by \autoref{lem:rank_b_n_minus_2}, 
    \[
        \rktwo(B) = \rktwo
            \renewcommand{\arraystretch}{1.5}\left[\begin{array}{c|c} 
            	C & C + W^\top \\ 
            	\hline 
            	C + W & C + W + W^\top  
            \end{array}\right] = n-2 
    \]
    which is a contradiction. $\qedhere$
    \end{proof}
    
    \begin{claim} \label{clm:xw_in_rowsp_c}
    If $x,w \in \rowsp(C)$, then $C$ is a twin free matrix of rank $(n-2)$.
    \end{claim}
    \begin{proof}
    We will prove this claim for Case I (using the notations from \autoref{fig:two_cases}), the proof for Case II is identical. From our assumption $x\in \rowsp(C)$, we deduce that there exist bits $b_w, b_x \in \{0,1\}$ such that,
    $$
    (x, x + b_w, x + b_x, x + b_w + b_x) \in \rowsp(C \mid C + W^\top \mid C + X^\top \mid C + W^\top + X^\top).
    $$
    Let us denote the top half of $A$ by $Q$. It suffices to show that $(\mbb{0}, \mbb{0}, \mbb{1}, \mbb{1}), (\mbb{0}, \mbb{1}, \mbb{0}, \mbb{1}), (\mbb{0}, \mbb{1}, \mbb{1}, \mbb{0}) \in \rowsp(Q)$, since this will imply $(x, \bar{x}, x, \bar{x}), (x, x, x, x) \in \rowsp(Q)$ and thus any row in the bottom half of $A$ is spanned by rows in the top half, which is a contradiction.
    
    Recall that from \autoref{lem:rank_b_n_minus_2}, $\rktwo(B) = n-2$. Now, note that for $b\in \{0,1\}$, twin rows in $B_{(x,x)^{-1}(b)}$ correspond to twin rows in $A$. Thus, following the same reasoning as in the proof of \autoref{lem:xw_rowsp_c}, we deduce that $x$ is \textit{balanced} ($|x^{-1}(0)| = |x^{-1}(1)| = 2^{n-3}$), and furthermore,
    \begin{align*}
        \rows\left(B_{(x,x)^{-1}(0)}\right) = \rows\left(B_{(x,x)^{-1}(1)}\right) = V \le \FF_2^{2^{n-1}}, \text{ where } \dim(V) = n-2.
    \end{align*}

    Assume for contradiction that some row $t\in \FF_2^{2^{n-2}}$ appears \textit{twice} in $C$, at indices $i\neq j$. It cannot be that $w_i = w_j$, since this would imply that $(t, t + w_i)$ appears \textit{four times} in $B$, in contradiction to its double-subspace structure outlined above. We remark that the rows $(t,t)$ and $(t, \overline{t})$ appear once in the indices of $x^{-1}(0)$ and once in the indices of $x^{-1}(1)$, and thus, without loss of generality, let $x_j = w_i = 1$, $x_i = w_j = 0$. This yields the following structure,
    \begin{figure}[H]
        \centering
        \input{drawings/no_odd_drawings/TikzSettings}

\def\toneheight{14.7}
\def\ttwoheight{10.7}
\def\offset{1.5}
\scalebox{0.9}{

    \begin{tikzpicture}[scale=.6, ->, thick]
    \draw[step=4cm,black,thick, -] (0,8) grid (16,16);
    \node (1) at (2,14) {{ \scalebox{.8}{${C}$}}};
    \node (2) at (6,14) {{ \scalebox{.8}{${C}{+}{W^\top}$}}};
    \node (3) at (2,10) {{ \scalebox{.8}{${C}{+}{W}$}}};
    \node (4) at (6,10) {{ \scalebox{.8}{${C}{+}{W}{+}{W^\top}$}}};
    \node (5) at (10,14) {{ \scalebox{.8}{${C}{+}{X^\top}$}}};
    \node (6) at (14,14) {{ \scalebox{.8}{${C}{+}{W^\top}{+}{X^\top}$}}};
    \node (7) at (10,10) {{ \scalebox{.8}{${C}{+}{X^\top}{+}{W}$}}};
    \node (8) at (14,10) {{ \scalebox{.8}{${C}{+}{W}{+}{W^\top}{+}{X^\top}$}}};
    
    \node (t1)[blue] at (2,\toneheight) {$t$};
    \node (t2)[blue] at (6,\toneheight) {$\overline{t}$};
    \node (t3)[blue] at (10,\toneheight) {$t$};
    \node (t4)[blue] at (14,\toneheight) {$\overline{t}$};
    
    \node (t5)[blue] at (2,\toneheight - \offset) {$t$};
    \node (t6)[blue] at (6,\toneheight - \offset) {$t$};
    \node (t7)[blue] at (10,\toneheight - \offset) {$t$};
    \node (t8)[blue] at (14,\toneheight - \offset) {$t$};
    
    \node (t21)[blue] at (2,\ttwoheight) {$t$};
    \node (t22)[blue] at (6,\ttwoheight) {$\overline{t}$};
    \node (t23)[blue] at (10,\ttwoheight) {$t$};
    \node (t24)[blue] at (14,\ttwoheight) {$\overline{t}$};
    
    \node (t25)[blue] at (2,\ttwoheight - \offset) {$t$};
    \node (t26)[blue] at (6,\ttwoheight - \offset) {$t$};
    \node (t27)[blue] at (10,\ttwoheight - \offset) {$\overline{t}$};
    \node (t28)[blue] at (14,\ttwoheight - \offset) {$\overline{t}$};
    
    \node (i)[black] at (-.3,\toneheight) {$i$};
    \node (j)[black] at (-.3,\toneheight - \offset) {$j$};

    \node (w1) [Transparent Orange, opacity=0.6] at (4.3,10) {$w$};
    \node (w2) [Transparent Orange, opacity=0.6] at (4.3,14) {$w$};

    \node (x1) [red] at (8.3,10) {$x$};
    \node (x2) [red] at (8.3,14) {$x$};

\draw[->, Transparent Orange, opacity=0.6]
    (w2) edge (4.3,15.9) edge (4.3,12.1)
    (w1) edge (4.3,11.9) edge (4.3,8.1);

\draw[->, red, opacity=0.6]
    (x2) edge (8.3,15.9) edge (8.3,12.1)
    (x1) edge (8.3,11.9) edge (8.3,8.1);

\draw[->, blue, thick]
    (t1) edge (.2,\toneheight) edge (3.8,14.7)
    (t2) edge (4.2,\toneheight) edge (7.8,14.7)
    (t3) edge (8.2,\toneheight) edge (11.8,\toneheight)
    (t4) edge (12.2,\toneheight) edge (15.8,\toneheight)
    
    (t5) edge (.2,\toneheight -\offset) edge (3.8,\toneheight -\offset)
    (t6) edge (4.2,\toneheight -\offset) edge (7.8,\toneheight -\offset)
    (t7) edge (8.2,\toneheight -\offset) edge (11.8,\toneheight -\offset)
    (t8) edge (12.2,\toneheight -\offset) edge (15.8,\toneheight - \offset)
    
    (t21) edge (.2,\ttwoheight) edge (3.8,\ttwoheight)
    (t22) edge (4.2,\ttwoheight) edge (7.8,\ttwoheight)
    (t23) edge (8.2,\ttwoheight) edge (11.8,\ttwoheight)
    (t24) edge (12.2,\ttwoheight) edge (15.8,\ttwoheight)
    
    (t25) edge (.2,\ttwoheight - \offset) edge (3.8,\ttwoheight - \offset)
    (t26) edge (4.2,\ttwoheight - \offset) edge (7.8,\ttwoheight - \offset)
    (t27) edge (8.2,\ttwoheight - \offset) edge (11.8,\ttwoheight - \offset)
    (t28) edge (12.2,\ttwoheight - \offset) edge (15.8,\ttwoheight - \offset);
 

\end{tikzpicture}    
}
    \end{figure}
    \noindent Hence indeed $(\mbb{0}, \mbb{0}, \mbb{1}, \mbb{1}), (\mbb{0}, \mbb{1}, \mbb{0}, \mbb{1}), (\mbb{0}, \mbb{1}, \mbb{1}, \mbb{0}) \in \rowsp(Q)$, as required. $\qedhere$
    \end{proof}
    
    It remains to handle the case in which $x,w \notin \rowsp(C)$ (and thus $\rktwo(C) = n-4)$. Note that if for some $b,b' \in \{0,1\}$ we have $\abs{x^{-1}(b)\cap w^{-1}(b')} > 2^{n-4}$, then $C_{x^{-1}(b)\cap w^{-1}(b')}$ has repeating rows, and so does $A$. Therefore all the intersections must be \textit{balanced},
    \[
        \abs{x^{-1}(0)\cap w^{-1}(0)} = \abs{x^{-1}(0)\cap w^{-1}(1)} = \abs{x^{-1}(1)\cap w^{-1}(0)} = \abs{x^{-1}(1)\cap w^{-1}(1)} = 2^{n-4}.
    \]
    Finally, for additional simplicity in what follows, let us reorder $w$ and $x$ \textit{simultaneously} by applying a permutation $\pi\in S_{2^{n-2}}$ under which $w = \rbk{\mbb{0}, \mbb{0}, \mbb{1}, \mbb{1}}$ and $x = \rbk{\mbb{0}, \mbb{1}, \mbb{0}, \mbb{1}}$. To this end, let us conjugate every block of $A$ using $\pi$, and denote the resulting top-left block of $C$ by $D$, see \autoref{fig:permutation_on_a} and \autoref{fig:higher_res_permutation}.
    
    \begin{figure}[H]
        \centering
        \input{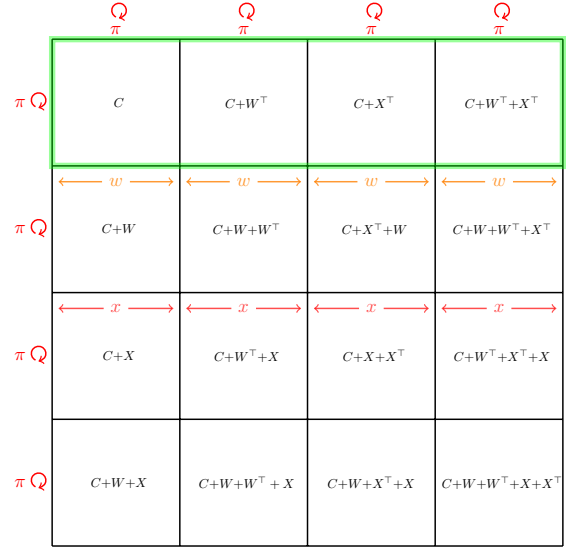}

\scalebox{.7}{

    \begin{tikzpicture}[scale=.6, ->, thick]
    \draw[step=4cm,black,thick, -] (0,0) grid (16,16);
    \node (1) [Low Visible] at (2,14) {{ \scalebox{.7}{${C}$}}};
    \node (2) [Low Visible] at (6,14) {{ \scalebox{.7}{${C}{+}{W^\top}$}}};
    \node (3) [Low Visible]at (2,10) {{ \scalebox{.7}{${C}{+}{W}$}}};
    \node (4) [Low Visible]at (6,10) {{ \scalebox{.7}{${C}{+}{W}{+}{W^\top}$}}};
    \node (5) [Low Visible]at (10,14) {{ \scalebox{.7}{${C}{+}{X^\top}$}}};
    \node (6) [Low Visible]at (14,14) {{ \scalebox{.7}{${C}{+}{W^\top}{+}{X^\top}$}}};
    \node (7) [Low Visible] at (10,10) {{ \scalebox{.7}{${C}{+}{X^\top}{+}{W}$}}};
    \node (8) [Low Visible] at (14,10) {{ \scalebox{.7}{${C}{+}{W}{+}{W^\top}{+}{X^\top}$}}};
    \node (9) [Low Visible]at (2,6) {{ \scalebox{.7}{${C}{+}{X}$}}};
    \node (10) [Low Visible] at (6,6) {{ \scalebox{.7}{${C}{+}{W^\top}{+}{X}$}}};
    \node (11)[Low Visible]  at (2,2) {{ \scalebox{.7}{${C}{+}{W}{+}{X}$}}};
    \node (12)[Low Visible] at (6,2) {{ \scalebox{.7}{${C}{+}{W}{+}{W^\top + X}$}}};
    \node (13) [Low Visible] at (10,6) {{ \scalebox{.7}{${C}{+}{X}{+}{X^\top}$}}};
    \node (14)[Low Visible] at (14,6) {{  \scalebox{.7}{${C}{+}{W^\top}{+}{X^\top}{+}{X}$}}};
    \node (15)[Low Visible] at (10,2) {{\scalebox{.7}{${C}{+}{W}{+}{X^\top}{+}{X}$}}};
    \node (16) [Low Visible] at (14,2) {{ \scalebox{.7}{${C}{+}{W}{+}{W^\top}{+}{X}{+}{X^\top}$}}};

    \node (w1) [Transparent Orange] at (2,11.5) {$w$};
    \node (w2) [Transparent Orange] at (6,11.5) {$w$};
    
    \node (s) [Transparent Orange] at (10,11.5) {$w$};
    \node (t) [Transparent Orange] at (14,11.5) {$w$};
    
    \node (x1) [Transparent Red] at (2,7.5) {$x$};
    \node (x2) [Transparent Red] at (10,7.5) {$x$};
    
    \node (y1) [Transparent Red] at (6,7.5) {$x$};
    \node (y2) [Transparent Red] at (14,7.5) {$x$};
    
    \draw[->, Transparent Orange]
        (w1) edge (.2,11.5) edge (3.8,11.5)
        (w2) edge (4.2,11.5) edge (7.8,11.5)
        (s) edge (8.2,11.5) edge (11.8,11.5)
        (t) edge (12.2,11.5) edge (15.8,11.5);
    \draw[->, Transparent Red]    
        (x1) edge (.2,7.5) edge (3.8,7.5)
        (y1) edge (4.2,7.5) edge (7.8,7.5)
        (x2) edge (8.2,7.5) edge (11.8,7.5)
        (y2) edge (12.2,7.5) edge (15.8,7.5);

    \draw[line width = 1.2mm, opacity = .4, green] (0,16) -- (16,16) -- (16,12) -- (0,12) -- cycle;
    \foreach \x in {2,6,10,14}{
    \node (\x) [red] at (\x, 16.3) {$\mathbf{\pi}$};
    \draw [->, shorten >=-1.5pt, red] (\x, 16.7) arc (245:-70:1.5ex);
    }
    
    \foreach \y in {2,6,10,14}{
    \node (\y) [red] at (-1,\y) {$\mathbf{\pi}$};
    \draw [->, shorten >=-1.5pt, red] (-0.52, \y - 0.13) arc (245:-70:1.5ex);
    }
    
    \end{tikzpicture}    
}
        \caption{Conjugating the matrix $A$ with the permutation $\tau = (\pi, \pi, \pi, \pi)$. Applying this permutation $\tau$ may destroy the \textit{internal} coset structure of the square submatrices of dimension $2^{n-2} \times 2^{n-2}$, but nevertheless preserves their \textit{external} structure. In what follows we rely only on this external structure. }
        \label{fig:permutation_on_a}
    \end{figure}
    \begin{figure}[H]
        \centering
        \input{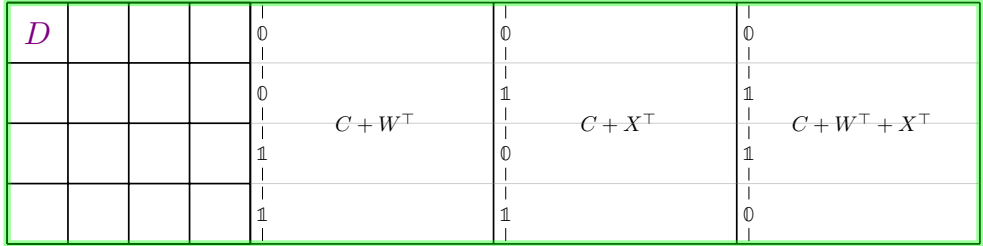}

\scalebox{.8}{
\begin{tikzpicture}

\draw[step=4cm,black,thick, -] (0,0) grid (16,4);
\draw[step=1cm,black,thick, -] (0,0) grid (4,4);
\draw[line width = 1.2mm, opacity = .4, green] (0,0) -- (0,4) -- (16,4) -- (16,0) -- cycle;
\node[violet] at (.5,3.5) {\scalebox{1.5}{$D$}};

\node (2) [Low Visible] at (6,2) {{ {$C + W^\top$}}};
\node (5) [Low Visible]at (10,2) {{  ${ C + X^\top}$}};
\node (6) [Low Visible]at (14,2) {{  ${ C + W^\top + X^\top}$}};

\foreach \y in {1,2,3}{
\draw[-,gray, opacity = .4] (4,\y) edge (16,\y);
}
\node (21) at (4.2,3.5) {\scalebox{1}{$\mbb{0}$}};
\node (22) at (4.2,2.5) {\scalebox{1}{$\mbb{0}$}};
\node (23) at (4.2,1.5) {\scalebox{1}{$\mbb{1}$}};
\node (24) at (4.2,0.5) {\scalebox{1}{$\mbb{1}$}};

\node (31) at (8.2,3.5) {\scalebox{1}{$\mbb{0}$}};
\node (32) at (8.2,2.5) {\scalebox{1}{$\mbb{1}$}};
\node (33) at (8.2,1.5) {\scalebox{1}{$\mbb{0}$}};
\node (34)at (8.2,0.5) {\scalebox{1}{$\mbb{1}$}};

\node (41) at (12.2,3.5) {\scalebox{1}{$\mbb{0}$}};
\node (42) at (12.2,2.5) {\scalebox{1}{$\mbb{1}$}};
\node (43) at (12.2,1.5) {\scalebox{1}{$\mbb{1}$}};
\node (44) at (12.2,0.5) {\scalebox{1}{$\mbb{0}$}};

\draw[-] 
(4.2,1.95) -- (23) -- (4.2,1.05)
(4.2,0.95) -- (24) -- (4.2,0.05)

(8.2,2.95) -- (32) -- (8.2,2.05)
(8.2,0.95) -- (34) -- (8.2,0.05)

(12.2,2.95) -- (42) -- (12.2,2.05)
(12.2,1.95) -- (43) -- (12.2,1.05)

;
\draw[-]
(4.2,3.95) -- (21) -- (4.2,3.05)
(4.2,2.95) -- (22) -- (4.2,2.05)

(8.2,3.95) -- (31) -- (8.2,3.05)
(8.2,1.95) -- (33) -- (8.2,1.05)

(12.2,3.95) -- (41) -- (12.2,3.05)
(12.2,0.95) -- (44) -- (12.2,0.05)
;
\end{tikzpicture}
}
        \caption{The internal structure of $C$, and the external structure of the top quarter of $A$.}
        \label{fig:higher_res_permutation}
    \end{figure}
    
    \begin{claim} \label{clm:xw_notin_rowsp_c}
    If $x,w\notin \rowsp(C)$, then $D$ is a twin free matrix of rank $(n-4)$.
    \end{claim}
    \begin{proof}
    To avoid repeated rows in $A$, we have
    \[
        \rows(C_{w^{-1}(0)\cap x^{-1}(0)}) = \dots = \rows(C_{w^{-1}(1)\cap x^{-1}(1)}) = V \le \FF^{2^{n-2}}_2, \text{ where } \dim(V) = n-4.
    \]
    That is, all the horizontal rectangles in $C$ are equal, thus:
    \begin{figure}[H]
        \centering
        \input{drawings/no_odd_drawings/TikzSettings}

\scalebox{.5}{
    \def\s{12}
    \def\ss{24}
    \begin{tikzpicture}[scale=.6, thick]

    \draw[step=2cm,black] (0,0) grid (8,8);
    \draw[step=2cm,black] (0 +\s,0) grid (8 + \s,8);
    \draw[step=2cm,black] (0 +\ss,0) grid (8 + \ss,8);
    
    \foreach \x in {0,\s,\ss}
    {\draw[line width = 2mm, opacity = .4, blue]
     (0 + \x,8) -- (8 + \x, 8) -- (8 + \x,0) -- (0 + \x,0) --cycle; 
    \foreach \y in {2,4,6}
    {
     \draw[line width = 2mm, opacity = .4, blue]
     (\x,\y) -- (\x + 8, \y);
    }
    }
    \foreach \y in {1,3,5,7}{
        \node (1) at (1,\y) {\scalebox{2}{$D$}};
    };
    \foreach \y in {1,3,5,7}{
        \node (1) at (1 +\s,\y) {\scalebox{2}{$D$}};
    };
    
    \foreach \x in {3,5,7}{
        \node (1) at (\x +\s,7) {\scalebox{2}{$D$}};
    };
    
    \foreach \y in {1,3,5,7}
        \foreach \x in {1,3,5,7}{
        \node (1) at (\x + \ss,\y) {\scalebox{2}{$D$}};
    };
    
    \node (2) at (10,4) {\scalebox{2}{$\Longrightarrow$}};
    \node[draw, align=center, draw=none] at (10,2.8) {Symmetry of \\ $C,D$};
    
    \node (2) at (22,4) {\scalebox{2}{$\Longrightarrow$}};
    \node[draw, align=center, draw=none] at (22,2.6) {All blue \\ rectangles \\ are equal};
    
    \node [black] (3) at (-2,4) {\scalebox{2}{$C \; = $}};
    
    \end{tikzpicture}    
}
    \end{figure}
    \noindent So $C = J_4 \otimes D$, where $J_4$ is the all-ones matrix (which is of rank-1), and hence by multiplicativity of the rank under the Kronecker product, 
    \[
        n-4 = \rktwo(C) = \rktwo(J_4)\cdot \rktwo(D) = \rktwo(D).
    \] 
    Finally, we observe that $D$ cannot have duplicate rows, as those would imply duplicates in $C_{x^{-1}(b)\cap w^{-1}(b')}$ (for any $b,b'\in \{0,1\}$), and in turn, in $A$ (which would be a contradiction). $\qedhere$
    \end{proof}   
    
    \autoref{lem:xw_rowsp_c}, alongside with \autoref{clm:xw_in_rowsp_c} and \autoref{clm:xw_notin_rowsp_c}, imply that $A$ contains a symmetric submatrix of size $2^m \times 2^m$, having no duplicate rows, zeros on its main diagonal, and rank $m\in \{ n-2, n-4\}$. However, we recall that no such matrices exist for $n \in \{1,3\}$ (see \autoref{clm:no_twinfree_n_3}), and thus by induction no such matrices exist for \textit{any} odd $n$. This concludes the proof of \autoref{thm:no_odd_n}. $\qedhere$
\end{proof}

\subsection{A Tight Construction for Odd $n$}

While \autoref{thm:no_odd_n} rules out the possibility of twin-free graphs of order $2^n$ and $\FF_2$-rank $n$ for all odd $n$, we remark that rank $(n+1)$ is indeed attainable, by utilizing the construction of \autoref{thm:infinite_family_twinfree_lowf2rank}.

\begin{corollary}
    For every odd $n$, there exists a \textit{twin-free} graph $G$ of order $2^{n}$ with $\FF_2$-rank $\le n+1$.
\end{corollary}
\begin{proof}
    Let $n = 2m + 1$ where $m \in \mathbb{N}$, and let $G = K_2 \boxplus G_2^{\boxplus m}$ where $K_2$ is an edge. Let $A$ be the adjacency matrix of $G_2^{\boxplus m}$ and let: 
    
    \[ A_G = \left[ \begin{array}{c|c}
                \rule{0pt}{1.05\normalbaselineskip} A & \bar{A} \\
                \hline
                \rule{0pt}{1.05\normalbaselineskip} \bar{A} & A
            \end{array} \right]
    \]
    
    be the adjacency matrix of $G$, where $\bar{A} = J + A$ (over $\FF_2$). By construction, the order of $G$ is $2^n$. Furthermore, by \autoref{lem:par_prod_properties}, we have that $\rktwo(G) \le \rktwo\left(K_2\right) + \rktwo\left(G_2^{\boxplus m}\right) = 2 + 2m = n + 1$. Finally, since $A$ is \textit{negation-free} and \textit{twin-free}, then $A_G$ is \textit{twin-free} (every row of $A_G$ is either $(v,\bar{v})$ or $(\bar{v}, v)$ for some row $v$ of $A$).
\end{proof}
\section{Future Work}
Lov\'asz and Kotlov \cite{kotlov1996rank} showed that, over the reals, any twin-free graph of order $2^n$ must have rank at least $2n - \mathcal{O}(1)$. In contrast, we show that over $\mathbb{F}_2$ no non-trivial bound holds, except up to an additive constant of $1$ for the case of odd $n$. Subsequently, it seems natural to ask the following intermediate question: what is the behaviour of rank-extremal twin-free graphs, where \textit{the rank is taken over a finite field larger than $\FF_2$}? Do non-trivial lower bounds hold in such cases, and conversely, can the extremal families of such cases be constructed or classified? \textit{These are left as open questions.}

\section{Acknowledgements}

We would like to thank Nati Linial for helpful discussions on this paper. We would also like to thank Sam Adriaensen for bringing to our attention the results of Godsil and Royle.

\bibliography{f2rank}
\bibliographystyle{alpha}

\end{document}